\documentclass{amsart}
\usepackage{amsfonts,latexsym,amsthm,amssymb,xspace}
\usepackage{amscd}
\usepackage{amsmath}

\usepackage{color}

\usepackage{graphicx, amssymb, amsmath}
\usepackage{amsthm,amsfonts}
\usepackage{url}
\usepackage{enumerate}
\usepackage{latexsym}

\theoremstyle{plain}
\newtheorem{theor}{Theorem}
\theoremstyle{plain}
\newtheorem{prop}{Proposition}
\theoremstyle{plain}
\newtheorem{lemma}{Lemma}
\theoremstyle{plain}
\newtheorem{cor}{Corollary}
\theoremstyle{definition}

\theoremstyle{remark}
\newtheorem{rem}{Remark}
\theoremstyle{remark}

\begin{document}

\title[]{Disjointly homogeneous rearrangement invariant spaces via interpolation}

\author{S.~V.~Astashkin}

\maketitle

\begin{abstract} A Banach lattice $E$ is called {\it $p$-disjointly homogeneous}, $1\le p\le\infty,$ 
when every sequence of pairwise disjoint normalized elements in $E$ has a 
subsequence equivalent to the unit vector basis of $\ell_p$. 
Employing methods from interpolation theory, we clarify which r.i. spaces on $[0,1]$ 
are $p$-disjointly homogeneous. In particular, for every $1<p<\infty$ 
and any increasing concave function $\varphi$ on $[0,1],$ which is not equivalent neither $1$ nor $t$,   
there exists a $p$-disjointly homogeneous r.i. space with the fundamental function $\varphi.$
Moreover, in the class of all interpolation r.i. spaces with respect
to the Banach couple of Lorentz and Marcinkiewicz spaces with the same fundamental function,
dilation indices of which are non-trivial, for every $1<p<\infty$, there is only a 
unique $p$-disjointly homogeneous space.
\end{abstract}

Primary classification: 46E30

Secondary classification(s): 46B03, 46B70







\def\rs{{\cal K}}
\def\pq{X_{\theta,p}}
\def\vr{{M(\tilde{\varphi})}}
\def\fg{{\mathbb N}}
\def\il{{\Lambda(\varphi)}}
\def\iod{i_1, \ldots ,i_d}
\def\ilk{i_1< \ldots <i_k}
\def\iok{i_1, \ldots ,i_k}
\def\ab{{\cal X}}
\def\bc{{\mathbb Z}}
\def\hj{{\mathbb R}}


\section{Introduction}

A Banach lattice $E$ is called \textit{disjointly homogeneous} if any two sequences $\{x_n\}_{n=1}^\infty$
and $\{y_n\}_{n=1}^\infty$ of pairwise disjoint normalized elements from $E$ contain subsequences $\{x_{n_k}\}$ and $\{y_{n_k}\}$,
respectively, which are equivalent in $E$. Similarly, $E$ will be called {\it $p$-disjointly homogeneous}, $1\le p\le\infty,$ 
when every sequence of pairwise disjoint normalized elements in $E$ has a 
subsequence equivalent to the unit vector basis of $\ell_p$. The
importance of disjointly homogeneous spaces introduced in \cite{FTT2} 
is based, mainly, on close connections between the classes of compact and strictly singular operators
acting in such spaces. 
Recall that a linear operator between Banach spaces is {\it strictly singular} if it is not an isomorphism when restricted to
any infinite-dimensional subspace. Recently, it was obtained a series of interesting results 
showing that a strictly singular operator has some compact power whenever the Banach lattice in
which it is bounded is disjointly homogeneous (see \cite{FHST}, \cite{FHSTT}, \cite{FTT2}, \cite{HST}); in particular, 
in the paper \cite{FHST}, it is proved that every strictly singular operator in
a $p$-disjointly homogeneous rearrangement invariant (r.i.) space with lower Boyd index 
$\alpha_X>0$ has compact square and if $p=2$ such a operator even is compact itself.
For this reason it is important to know how wide the class
of disjointly homogeneous Banach lattices. As is shown in the above cited papers
it contains $L_p(\mu)$-spaces, $1\le p\le \infty,$
Lorentz function spaces $L_{q,p}$ and $\Lambda(W, p)$, certain classes of Orlicz function
spaces and also some discrete spaces such as the Tsirelson space. 

The main aim of this paper is to clarify which r.i. spaces on $[0,1]$ 
are $p$-disjointly homogeneous. We focused on the more interesting reflexive case, when $1< p<\infty.$ 
Our approach to this problem is based on using tools from interpolation
theory, especially, the real and complex methods of interpolation. By the complex method
of interpolation, we prove that for every $1<p<\infty$ and any increasing concave function 
$\varphi$ on $[0,1],$ which is not equivalent neither $1$ nor $t$,   
there exists a $p$-disjointly homogeneous r.i. space with the fundamental function $\varphi$
(Corollary \ref{main corollary}). Note that there is the only r.i. space
on $[0,1]$, $L_\infty$ (resp. $L_1$), having the fundamental function 
equivalent to $1$ (resp. $t$). This result is new even for the power functions
$\varphi(t)=t^{\alpha},$ $0<\alpha<1.$ Moreover, in the class of all interpolation r.i. spaces with respect
to the Banach couple of Lorentz and Marcinkiewicz spaces with the same fundamental function
$\varphi$, dilation indices of which are non-trivial, for every $1<p<\infty$, there is only 
a unique $p$-disjointly homogeneous space, 
namely, the Lorentz space $\Lambda_{p,\varphi}$ 
with the quasi-norm $\|x\|_{p,\varphi}= (\int_{0}^{1} [ x^{*}(t)
\varphi(t)]^{p}\,{dt}/{t})^{1/p}$ (Theorem \ref{Uniqueness}).
At the same time, in Section 4, for every $1<p<\infty$ and any increasing concave function $\varphi$ on
$[0,1]$ such that $\lim_{t\to 0}\varphi(t)=0$ and upper dilation index $\beta_\varphi<1,$
we construct a $p$-disjointly homogeneous r.i. 
space with the fundamental function $\varphi$, which is not interpolation with respect to the 
corresponding couple of Lorentz and Marcinkiewicz spaces (Theorem \ref{theor 5}). Finally,
in Section 5 we investigate some properties of sequences of pairwise disjoint functions in the 
real interpolation spaces $(X_0,X_1)_{\theta,p}$ $(0<\theta<1,$ $1\le p<\infty)$
provided that $X_1\subset X_0.$

The author is grateful to Prof. E.M. Semenov for bringing this problem 
and very useful discussions related to it.

\section{Preliminaries}


\subsection{Rearrangement invariant spaces}
In this subsection we present some definitions and auxiliary results from the theory of
 rearrangement invariant spaces. For more details on the latter theory
 we refer to \cite{BS, KPS, LT2}.
 
A Banach function space $X = (X, \| \cdot \|)$ of (classes of) real measurable functions $x(t)$ 
defined on the interval $[0, 1]$ is said to be {\it rearrangement 
invariant (r.i.) space} if the conditions $x^*(t) \leq y^*(t)$ a.e. on $[0,1]$ 
and $y \in  X$ imply $x \in X$ and $\| x \|_{X} \leq \| y \|_{X}$. 
Here, $x^*$ denotes the non-increasing right-continuous rearrangement of $|x(s)|$
given by
$$
x^{*}(t)=\inf \{~\tau\ge 0:m (\{s\in [0,1]:\,|x(s)|>\tau\})\le t~\},\quad 0\le t\le 1,
$$
where  $m $ is the Lebesgue measure.

For every r.i. space $X$ on $[0,1]$ we have the continuous embeddings $ L_{\infty }[0,1] \subset X \subset 
L_{1}[0,1]$. The {\it fundamental function} of an r.i. space $X$ is given by
$\varphi _{X}(t): = \| \chi_{[0, t]}\|_{X},$ $0\le t \le 1$, where 
$\chi _{A}$ denotes the characteristic function of a measurable set $A\subset[0,1]$.
It is well known that every fundamental function is quasi-concave 
on $[0,1],$ i.e., it is non-decreasing and the function ${\varphi_X(t)}/{t}$ 
is non-increasing on $(0, 1]$. Each quasi-concave function $\varphi$ on $[0,1]$  is equivalent
to its least concave majorant $\bar \varphi$, more exactly,
$\frac{1}{2}\bar \varphi(t) \le \varphi(t)\le\bar\varphi(t)\;\;(0\le t\le 1)$ \cite[Theorem~2.1.1]{KPS}.

If $X$ is an r.i. space on $[0,1]$, then {\it the
K\" othe dual} space
$X'$ consists of all measurable functions $y$ such that
$$
\|y\|_{X'}\,\,=\,\,\sup\,\biggl\{\int_{0}^1{x(t)y(t)\,dt}:\;\;
\|x\|_{X}\,\leq{1}\biggr\}\,<\,\infty.
$$
The space $X'$ is r.i. as well; it is embedded into the dual space
$X^*$ of $X$ isometrically, and $X'=X^*$ if and only if $X$ is
separable. An r.i. space $X$ is said {\it to have the Fatou
property} if the conditions $x_n\in X$ $(n=1,2,\dots),$
$\sup_{n=1,2,\dots}\|x_n\|_X<\infty$, and $x_n\to{x}$ a.e. imply
that $x\in X$ and $||x||_X\le \liminf_{n\to\infty}{||x_n||_X}.$  $X$
has the Fatou property if and only if the natural embedding of $X$
into its second K\" othe dual $X''$ is an isometric surjection. 

For a given $t > 0$ the {\it dilation operator} $\sigma_{t}$ defined by 
$\sigma_{t} x(s) = x(s/t) \chi_{[0,1]}(s/t), s \in [0,1],$ is bounded in every 
r.i. space $X$ and $\| \sigma_{t}\|_{X \rightarrow X} \leq max(1, t)$. 
The {\it lower} and {\it upper Boyd indices} of $X$ are defined by
$$
\alpha_{X} = \lim_{t \rightarrow 0^+} \frac {\log \| \sigma_{t}\|_{X \rightarrow 
X}}{\log t},
~~~\beta_{X} = \lim_{t \rightarrow \infty } \frac {\log \| \sigma_{t}\|_{X \rightarrow 
X}}{\log t},
$$
respectively. In general, $0 \leq \alpha_{X} \leq \beta_{X} \leq 1$. 

The {\it lower} and {\it upper dilation indices} of an increasing concave 
function $\psi: [0, 1] \rightarrow [0, \infty)$ are defined as 
\begin{equation*}\label{I:Def1}
\gamma_{\psi} = \lim_{t \rightarrow 0^{+}} \frac {\log m_{\psi}(t)}{\log t},
~~~ \delta_{\psi} = \lim_{t \rightarrow \infty } \frac {\log m_{\psi}(t)}{\log 
t}, ~{\rm where} ~~m_{\psi}(t) = \sup_{0 < s < 1, 0 < st < 1} \frac {\psi(st)}{\psi(s)},
\end{equation*}
respectively. We have $0 \leq \gamma_{\psi} \leq \delta_{\psi} \leq 1$ and from 
the inequality $m_{\varphi_{X}}(t) \leq \| \sigma_{t}\|_{X \rightarrow X},$ 
$t > 0,$ it follows that $ \alpha_{X} \leq \gamma_{\varphi_{X}} \leq 
\delta_{\varphi_{X}} \leq \beta_{X}$.

Most important examples of r.i. spaces are the $L_p$--spaces $(1 \leq p \leq
\infty)$ and their natural generalization, the Orlicz spaces.
The {\it Orlicz space} $L_{M}$ on $[0, 1]$ is generated by the 
Luxemburg-Nakano norm
$$
\|x\|_{L_{M}} = \inf \Big\{\lambda > 0: \int_{0}^{1} 
M\Big(\frac{|x(t)|}{\lambda}\Big)\, dt \leq 1 \Big\},
$$
where $M$ is an Orlicz function, that is, an increasing convex 
function on $[0, \infty)$ such that $M(0) = 0$. The fundamental 
function of $L_M$ is $\varphi_{L_{M}}(t) = {1}/{M^{-1}({1}/{t})}$.

An important role will be played throughout the paper by Lorentz and Marcinkiewicz spaces. 
Let $\varphi$  be an increasing concave 
function on $[0,1]$. The norms of the {\it Lorentz} $\Lambda(\varphi)$ and 
{\it Marcinkiewicz spaces} $M(\varphi)$ are defined by the 
functionals
$$
\|x\|_{\Lambda(\varphi)} = \int_0^1 x^{*}(t)\, d\varphi(t) ~~{\rm and}
~~\|x\|_{M(\varphi)} = \sup_{0 < t \leq 1} \frac{1}{\varphi(t)}\int_{0}^{t} x^{*}(s)\, ds,
$$
respectively. The spaces $\Lambda(\varphi)$ and $M(\tilde{\varphi}),$ where $\tilde{\varphi}(t):=t/\varphi(t),$
have the same fundamental function equal to $\varphi.$ Moreover, these spaces are 
extreme in the class of all r.i. spaces with the fundamental function $\varphi,$ i.e.,
if $\varphi_X(t)=\varphi(t),$ $0\le t\le 1,$ then  $\Lambda(\varphi)\subset X\subset M(\tilde{\varphi})$ and
\begin{equation}\label{eq99}
 \|x\|_{M(\tilde{\varphi})}\le \|x\|_X\le \|x\|_{\Lambda(\varphi)}\;\;(x\in \Lambda(\varphi))
\end{equation}
(see e.g. \cite[Theorems~2.5.5 and 2.5.7]{KPS}).

Recall also the definition of some generalizations of the
classical Lorentz spaces $\Lambda(\varphi)$. Let us begin with 
the most important for us Lorentz space $\Lambda_{p,\varphi}$, which will 
arise later as a $p$-disjointly homogeneous
space. For $1 \leq p < \infty$ and any increasing concave function 
$\varphi$ on $[0, 1]$ this space is generated by the functional
$$
\|x\|_{\Lambda_{p,\varphi}} = \left( \int_{0}^{1} \left[ x^{*}(t) \varphi(t) 
\right ]^{p} \frac{dt}{t} \right)^{1/p}.
$$
The spaces $\Lambda(\varphi)$ were investigated by Sharpley \cite{Sh} and Raynaud \cite{Ra}.
If $\gamma_{\varphi} > 0$, then  
$$
\varphi(t)\leq \int_{0}^{t} \frac{\varphi(s)}{s} ds \leq C \varphi(t),\;\;0 < t \leq 1$$  
(cf. \cite{KPS}, Lemma 2.1.4), and therefore $\Lambda_{1,\varphi}$ 
coincides with the Lorentz space $\Lambda(\varphi)$.
Moreover, $\Lambda_{p,\varphi}$ is an r.i. space on $[0,1]$ with the equivalent norm 
$$
\|x\|_{\Lambda_{p,\varphi}}^{\star} = \left( \int_{0}^{1} \left[ x^{**}(t) \varphi(t) 
\right ]^{p} \frac{dt}{t} \right)^{1/p}
$$
whenever $0 < \gamma_{\varphi} \leq \delta_{\varphi} < 1$ (cf. \cite{Sh}, Lemma 3.1), 
where $x^{**}(t)=\frac{1}{t}\int_0^t x^*(s)\,ds$,
and its fundamental function 
$\varphi_{\Lambda_{p,\varphi}}$ is equivalent to $\varphi$. 
If $\varphi(t)=t^{1/q},$ $1<q<\infty,$ and $1 \le p < \infty$ 
we obtain the well-known spaces $L_{q,p}$,
$$
\|x\|_{q,p}=\left( \int_{0}^{1} \left(t^{1/q} x^{*}(t)\right )^{p} \frac{dt}{t} \right)^{1/p},
$$
which are very important in interpolation theory of operators (see e.g. \cite{BS} and \cite{KPS}).


Much earlier, in \cite{Lo}, Lorentz introduced the another space of such a type,
$\Lambda (W,p),$ generated by the functional
$$
\|x\|_{\Lambda (W,p)} = \left(\int_0^1 [x^{*}(t)]^{p} 
W(t)\,dt \right)^{1/p},$$
where $1\le p<\infty$ and $W$ is a 
positive, non-increasing function on $[0,1]$ such that $\lim_{t\to 0}W(t)=+\infty$ 
and $\int_0^1W(t)\,dt=1$. This space is also $p$-disjointly homogeneous
(see Introduction) but, unlike the spaces $\Lambda_{p,\varphi}$, for a fixed
$1<p<\infty$ the fundamental function of $\Lambda (W,p)$ cannot be equivalent to arbitrary
increasing quasi-concave function. Indeed, $\varphi_{\Lambda (W,p)}(t)=(\int_0^tW(s)\,ds)^{1/p},$
$0\le t\le 1,$ and hence for every function $W$ with the above properties we have the following restriction on 
its upper dilation index: $\delta_{\varphi_{\Lambda (W,p)}}\le 1/p$.
In particular, this class of spaces does not embrace the spaces $L_{q,p}$ if $1<q<p<\infty.$

\medskip 

\subsection{Interpolation spaces and functors}
Let us recall some definitions from the theory of interpolation of operators;
more detailed information see in the monographs \cite{BS, BL, BK, KPS, Trib}.

Say that two Banach spaces $X_0$ and $X_1$ form a {\it Banach couple} $\bar{X} = (X_0, X_1)$ if they linearly and
continuously embedded into a Hausdorff topological linear space; then we can define
their sum $X_0 + X_1$ and intersection $X_0 \cap X_1$ with natural norms.
A Banach space $X$ is called an {\it intermediate space} between $X_0$ and $X_1$ if $X_0 \cap X_1  \subset X  \subset X_0 + X_1$. 
Such a space $X$ is called an {\it interpolation space} with respect to the couple $(X_0,X_1)$
(we write: $X\in Int(X_0, X_1)$) if, for any bounded linear operator $T: X_0 + X_1 \rightarrow X_0 + X_1$ such that the
restriction $T_{| X_i}: X_i \rightarrow X_i$ is bounded for $i = 0, 1$, the restriction $T_{|_X}: X \rightarrow X$ is also bounded 
and $\| T\|_{X \rightarrow X} \leq C \, \max \, \{ \| T\|_{X_0 \rightarrow X_0} , \| T\|_{X_1 \rightarrow X_1} \}$ 
for some $C >0\geq 1$. 

An {\it interpolation method} or {\it interpolation functor} ${\mathcal F}$ is a 
construction (a rule) which assigns to every Banach couple $\bar{X} = (X_0, X_1)$ 
an interpolation space ${\mathcal F}(\bar{X})$ with respect to it. 
One of the most important interpolation methods is the {\it $K$-method} known also as the 
{\it real Lions-Peetre interpolation method}. Throughout the paper we will consider 
only embedded Banach couples $(X_0,X_1)$ or, more precisely,
satisfying the condition: $ X_{1} \subset X_{0}$ and $\|x\|_{X_{0}} \leq\|x\|_{X_{1}}$ 
for all $x \in X_{1}$. For this reason, we adopt our definition of the real $K$-method 
to this special case (for general situation we refer the reader to monograph \cite{BK}).
For a Banach couple
$\bar{X} = (X_0, X_1)$ the {\it Peetre K-functional} of an element $f \in X_0+X_1$ is defined by
$$
K(t, f; X_0, X_1) = \inf \{ \| f_0\|_{X_0} + t \| f_1\|_{X_1}: f = f_0 + f_1, f_0 \in X_0, f_1 \in X_1 \},\;\;t>0.
$$
Let $ X_{1} \subset X_{0}$ and $\|x\|_{X_{0}} \leq\|x\|_{X_{1}},$ 
$x \in X_{1}$. Then, if $E$ is a Banach lattice of sequences $(a_k)_{k=1}^\infty$ such that
$l_1\subset E\subset l_\infty$ and $\varphi$ is a continuous increasing concave function on $[0,1]$
such that $\varphi(0)=0$, the space $(X_0, X_1)_{E(1/\varphi(2^{-n}))}^K$ of the $K$-method of interpolation consists of all
$f\in X_0+X_1$ such that 
$$
\left(K(2^{-n},f;X_0,X_1)\right)_{n=1}^\infty\in E(1/\varphi(2^{-n}))\;\;\mbox{and}\;\;
\| f\|_{\varphi,E}^K: = \| \left(K(2^{-n}, f; X_0, X_1)\right)\|_{E(1/\varphi(2^{-n}))}.$$
In what follows, if $E$ is a Banach sequence lattice and $b_n\ge 0$, by $E(b_n)$ we denote the Banach
lattice consisting of all sequences $(x_n)_{n=1}^\infty$ such that $(x_nb_n)\in E$ with the 
norm $\|(x_n)\|_{E(b_n)}:=\|(x_nb_n)\|_E.$ It is easy to see that the space  
$(X_0, X_1)_{E(1/\varphi(2^{-n}))}^K$ belongs to the class $Int(X_0,X_1)$
for every Banach couple $(X_0,X_1).$ In particular, if $E=l_p$, $1\le p<\infty$, and $\varphi(t)=t^{\theta},$ $0<\theta<1,$ 
we obtain classical {\it Lions--Peetre interpolation spaces} $(X_0,X_1)_{\theta,p}$ 
with the norm 
\begin{equation}\label{eq1000}
 \| f \|_{\theta, p} = \Big(\sum_{n=1}^\infty \left(2^{n\theta} K(2^{-n}, f; X_0, X_1)\right)^p\Big)^{1/p}.
\end{equation}

%

The complex method of interpolation
introduced by Calderon and Lions (see e.g. \cite[Chapter~4]{BL}) 
may be regarded as another important way of constructing
interpolation spaces. It associates to any Banach couple $(X_0,X_1)$
the family of spaces $[X_0,X_1]_\theta$ $(0\le\theta\le 1)$ interpolation with respect to this couple.
Next, we will consider this method only for couples of the form $(M(\tilde{\varphi}),\Lambda(\varphi)),$
where $\varphi$ is an increasing concave function on $[0,1],$  
using the fact that, for every $0<\theta<1$, the space $[M(\tilde{\varphi}),\Lambda(\varphi)]_{\theta}$ 
isometrically coincides with the Calder\'on-Lozanovskii
space $M(\tilde{\varphi})^{1-\theta}\Lambda(\varphi)^\theta$ \cite[Theorem~4.1.4]{KPS}. 
The latter one can be defined as follows:
if $(X_0,X_1)$ is a couple of r.i. spaces on $[0,1]$, 
the space $X_0^{1-\theta}X_1^\theta$ consists of all measurable
functions $x(t)$ on $[0,1]$ such that 
$$
|x(t)|\le C|y(t)|^{1-\theta}|z(t)|^\theta,\;\;0\le t\le 1,$$
for some $C>0,$ $y\in X_0$ and $z\in X_1$, with $\|y\|_{X_0}\le 1$,
$\|z\|_{X_1}\le 1$. The space $X_0^{1-\theta}X_1^\theta$ is endowed 
with the norm $\|x\|:=\inf C,$ where infimum is taken over all $C$
satisfying the preceding inequality together with some $y$ and $z$.

%
%

\medskip 
\subsection{Ultrasymmetric spaces} 
The following concept from the paper \cite{Pust} will be useful for us. 
An r.i. space $X$ on $[0, 1]$ is called {\it ultrasymmetric} if 
there exists an increasing concave function $\varphi$ on $[0, 1]$ such 
that $X \in$ Int($ M(\tilde{\varphi}),\Lambda(\varphi)$). 
It is clear that the latter assumption implies immediately that the fundamental function of $X$ 
is equivalent to $\varphi$, that is, there are constants $c_{0}, 
c_{1} > 0$ such that $c_{0} \varphi(t) \leq \varphi_{X}(t) \leq 
c_{1} \varphi(t)$ for all $t \in [0, 1]$. Ultrasymmetric spaces embrace all possible 
generalizations of Lorentz-Zygmund spaces and this class of spaces is useful in many 
applications (see, for example, \cite{Pust1} and \cite{Pust2}).
We mention also that the natural question, which Orlicz spaces are 
ultrasymmetric, was answered in the paper \cite{AM08}. 

In Section 4 we apply the following characterization theorem of 
ultrasymmetric spaces (cf. \cite[Theorem~2.1]{Pust}). 
Let us write it in the form used further: {\it An r.i. 
space $X$ such that $\gamma_{\varphi_{X}} > 0$ is ultrasymmetric if and only if}
$$
X = \left(L_{1}, L_{\infty} \right)_{E(\varphi_X(2^{-n}) 
2^{n})}^{K}, ~~{\it where} ~~E \in {\it Int}(l_{\infty}, l_{1}).
$$
In particular, clearly, if an r.i. space $X$ is ultrasymmetric, the Banach lattice
$E$ in the preceding formula should be symmetric. 


\section{Ultrasymmetric $p$-disjointly homogeneous r.i. spaces}


\begin{theor}\label{theor 3}
Let $\varphi$ be an increasing concave function on
$[0,1]$ such that $\varphi(0)=0$ and
\begin{equation}\label{eq1001}
 \lim_{t\to 0}\varphi(t)=\lim_{t\to 0}\frac{t}{\varphi(t)}=0.
\end{equation}

Then, for every $1<p<\infty,$ the space ${\mathcal L}_{[\theta]}:=[M(\tilde{\varphi}),\Lambda(\varphi)]_{\theta},$ 
with $\theta:=1/p$, is $p$-disjointly homogeneous.
\end{theor}
\begin{proof}
Let $x_n\in {\mathcal L}_{[\theta]}$ be pairwise disjoint functions, $x_n\ge 0,$ and
$\|x_n\|_{[\theta]}=1$ $(n=1,2,\dots).$
We begin with proving the upper estimate showing that some subsequence of $\{x_n\}$ (we will denote
it, as above, by $\{x_n\}$) satisfies the inequality
\begin{equation}\label{eq17}
 \Big\|\sum_{k=1}^\infty a_kx_k\Big\|_{[\theta]}\le C\|(a_k)\|_{l_p}
\end{equation}
for any sequence $(a_k)\in l_p.$

Firstly, since the space ${\mathcal L}_{[\theta]}$ coincides isometrically with the Calder\'on-Lozanovskii
space $M(\tilde{\varphi})^{1-\theta}\Lambda(\varphi)^\theta,$ $0<\theta<1$ 
(see \cite[Theorem~4.1.4]{KPS} or Preliminaries), for arbitrary $n\in\fg$ and $\varepsilon>0$ there are
$y_n\in\Lambda(\varphi),$ $z_n\in M(\tilde{\varphi}),$ $y_n\ge 0,$ $z_n\ge 0,$
${\rm supp}\,y_n\subset {\rm supp}\,x_n,$ ${\rm supp}\,z_n\subset {\rm supp}\,x_n$ 
such that $x_n\le (1+\varepsilon)y_n^{\theta}z_n^{1-\theta}$ and 
$\|y_n\|_{\Lambda(\varphi)}=\|z_n\|_{M(\tilde{\varphi})}=1.$
Taking into account \eqref{eq1001} and \cite[Theorem~5.1]{FJT} (see also \cite[Proposition~1]{Tok}), we can assume (passing
to a subsequence if necessary) that
\begin{equation}\label{eq15}
 \Big\|\sum_{k=1}^\infty b_ky_k\Big\|_{\Lambda(\varphi)}\asymp \|(b_k)\|_{l_1}\;\;\mbox{and}\;\;
 \Big\|\sum_{k=1}^\infty c_kz_k\Big\|_{M(\tilde{\varphi})}\asymp \|(c_k)\|_{l_\infty}
\end{equation}
for all $(b_k)\in l_1$ and $(c_k)\in l_\infty.$

Next, let $(a_k)\in l_p$, $a_k\ge 0$ $(k\in\fg),$ be arbitrary. Since
$\theta=1/p,$ we have $l_\infty^{1-\theta}l_1^{\theta}=l_p$ (see e.g.
\cite[Ch.~IV, the end of $\S\,1$]{KPS}). Hence,
there exist $(b_k)\in l_1$ and $(c_k)\in l_\infty,$ $b_k\ge 0$, $c_k\ge 0,$ 
satisfying the conditions: $\|(b_k)\|_{l_1}\le \|(a_k)\|_{l_p},$ $\|(c_k)\|_{l_\infty}\le \|(a_k)\|_{l_p},$ and
$a_k\le (1+\varepsilon)b_k^{\theta}c_k^{1-\theta}.$
Then, from the inequality
$$
\sum_{k=1}^\infty a_kx_k\le (1+\varepsilon)^2\Big(\sum_{k=1}^\infty b_ky_k\Big)^{\theta}
\Big(\sum_{k=1}^\infty c_kz_k\Big)^{1-\theta},$$
and \eqref{eq18} it follows that
$$
 \Big\|\sum_{k=1}^\infty a_kx_k\Big\|_{[\theta]}\le (1+\varepsilon)^2
\Big\|\sum_{k=1}^\infty b_ky_k\Big\|_{\Lambda(\varphi)}^{\theta}
\Big\|\sum_{k=1}^\infty c_kz_k\Big\|_{M(\tilde{\varphi})}^{1-\theta}\asymp
\|(b_k)\|_{l_1}^{\theta}\|(c_k)\|_{l_\infty}^{1-\theta}\le \|(a_k)\|_{l_p},
$$
and \eqref{eq17} is proved.

The reverse inequality we prove by using duality. Let us find 
$x_n'\in {\mathcal L}_{[\theta]}'$ such that $\|x_n'\|_{{\mathcal L}_{[\theta]}'}\asymp 1,$
${\rm supp}\,x_n'\subset {\rm supp}\,x_n,$ and
\begin{equation}\label{eq16c}
 \int_0^1 x_n x_n'\,dt=1,\;\;n\in\fg.
\end{equation}
From the K\" othe duality theorem (see e.g. \cite[Theorem~2]{Loz}
or \cite{Shes}) and the equalities $M(\psi)'=\Lambda(\psi)$ and $\Lambda(\psi)'=M(\psi)$, 
which hold for every increasing concave function $\psi$ \cite[Theorems~4.5.2 and 4.5.4]{KPS},
it follows that
$$
{\mathcal L}_{[\theta]}'=[M(\tilde{\varphi}),\Lambda(\varphi)]_{[\theta]})'=[M({\varphi}),\Lambda(\tilde{\varphi})]_{[1-\theta]}.$$
If $1/p+1/q=1,$ then $1-\theta=1/q.$ Therefore, as is shown above, there is a subsequence 
of the sequence $\{x_n'\}$ (we keep the same notation), which satisfies the inequality:
\begin{equation*}\label{eq17c}
 \Big\|\sum_{k=1}^\infty b_kx_k'\Big\|_{{\mathcal L}_{[\theta]}'}\le C\|(b_k)\|_{l_q}
\end{equation*}
for every $(b_k)\in l_q.$ Equivalently, the operator $T(b_k):=\sum_{k=1}^\infty b_kx_k'$
acts boundedly from $l_q$ into ${\mathcal L}_{[\theta]}'$. Since for arbitrary
$(b_k)\in l_q$ and $x\in {\mathcal L}_{[\theta]}''={\mathcal L}_{[\theta]}$ we have
$$
\int_0^1 x(t)T(b_k)(t)\,dt=\sum_{k=1}^\infty b_k\int_0^1 x(t)x_k'(t)\,dt,$$
the  K\" othe dual operator to $T$ is defined by the formula:
$T'x:=\left(\int_0^1 x(t)x_k'(t)\,dt\right)_{k=1}^\infty,$ and $T'$ is bounded from 
${\mathcal L}_{[\theta]}$ into $l_p.$ By \eqref{eq16c}, we have
$T'\left(\sum_{k=1}^\infty a_kx_k\right)=(a_k).$ Hence,
$$
\|(a_k)\|_{l_p}\le \|T'\|\Big\|\sum_{k=1}^\infty a_kx_k\Big\|_{{\mathcal L}_{[\theta]}},
$$
and desired result follows.
\end{proof}

Suppose that an increasing concave function $\varphi$ on $[0,1]$ does not satisfy the condition
\eqref{eq1001}. Obviously, then either $\varphi(t)\asymp 1$ or $\varphi(t)\asymp t$ on
$(0,1].$ A unique (up to equivalence of norms) r.i. space with the fundamental function $1$ (resp. $t$) 
is $L_\infty$ (resp. $L_1$). Since the first of them is $\infty$-disjointly homogeneous and the second one
is $1$-disjointly homogeneous, we obtain the following assertion. 

\begin{cor}\label{main corollary}
Let $1<p<\infty$ and $\varphi$ be an increasing concave function on
$[0,1],$ $\varphi(0)=0.$ Then, there exists a $p$-disjointly homogeneous r.i. space $X$ with 
the fundamental function $\varphi_X=\varphi$ if and only if $\varphi(t)$ is equivalent neither $1$ nor $t$.
\end{cor}

Clearly, $[M(\tilde{\varphi}),\Lambda(\varphi)]_{\theta}$ is an interpolation space with respect to
the couple $(M(\tilde{\varphi}),\Lambda(\varphi)),$ i.e., is ultrasymmetric.
Let us show that if the dilation indices of an increasing concave function $\varphi$ are non-trivial, 
 i.e., $0<\gamma_\varphi\le\delta_\varphi<1,$ for every $1<p<\infty$, there is a unique ultrasymmetric $p$-disjointly 
 homogeneous r.i. space $X$ with the fundamental function $\varphi$.


\begin{theor}\label{Uniqueness}
Let $1<p<\infty$ and let $\varphi$ be an increasing concave function on
$[0,1]$ such that $0<\gamma_\varphi\le\delta_\varphi<1.$
Then the Lorentz space $\Lambda_{p,\varphi}$ with the quasi-norm
$$
\|x\|_{\Lambda_{p,\varphi}} = \left( \int_{0}^{1} \left[ x^{*}(t)
\varphi(t) \right ]^{p} \frac{dt}{t} \right)^{1/p}
$$
is a unique $p$-disjointly homogeneous r.i. space from the set
$Int(M(\tilde{\varphi}),\Lambda(\varphi))$.
\end{theor}

To prove this result we will need Theorem 2 from \cite{A2} restated below as Proposition
\ref{Uniqueness1}.

\begin{prop}\label{Uniqueness1}
Suppose that sets $e_k\subset [0,1]$, $k=1,2,\dots,$ satisfy the conditions:
$$
0<m(e_{k+1})\le m(e_{k})\;\;(k=1,2,\dots),\;\;\sum_{k=1}^\infty m(e_{k})\le 1,$$
and
$$
\sup_{k=1,2,\dots}\frac{1}{ m(e_{k})}\sum_{i=k}^\infty m(e_{i})<\infty.$$
Then, if $X_1$ and $X_2$ are r.i. spaces, which are interpolation with respect to the couple
$(M(\tilde{\varphi}),\Lambda(\varphi)),$ with
$0<\gamma_\varphi\le\delta_\varphi<1,$ then from the equivalence
$$
\Big\|\sum_{k=1}^\infty c_k\chi_{e_{k}}\Big\|_{X_1}\asymp 
\Big\|\sum_{k=1}^\infty c_k\chi_{e_{k}}\Big\|_{X_2},\;\;\mbox{for all}\;\;c_k\in\mathbb R,$$
it follows that $X_1=X_2$ (with equivalence of norms).
\end{prop}

\begin{proof}[Proof of Theorem \ref{Uniqueness}]
Since $0<\gamma_\varphi\le\delta_\varphi<1,$ we have
$$
\Lambda(\varphi)=(L_1,L_\infty)_{l_1(2^{k}\varphi(2^{-k}))}^K\;,
\;M(\tilde{\varphi})=(L_1,L_\infty)_{l_\infty(2^{k}\varphi(2^{-k}))}^K$$
(see e.g. \cite{DKO}). For the same reason, the discrete Calderon operator is bounded in the space 
$l_p(2^{k}\varphi(2^{-k}))$ for every $1\le p\le\infty$ (see e.g. \cite[Theorem~3]{A3}).
Therefore, by the Brudnyi theorem \cite[Theorem~4.3.1]{BK}, we obtain
$$
{\mathcal L}_{[\theta]}=[M(\tilde{\varphi}),\Lambda(\varphi)]_{\theta}=(L_1,L_\infty)_{[l_\infty,l_1]_{\theta}(2^{k}\varphi(2^{-k}))}^K.$$
If $\theta=1/p$, as above, we have $[l_\infty,l_1]_{\theta}=l_\infty^{1-\theta}l_1^{\theta}=l_p.$ Hence,
$$
{\mathcal L}_{[\theta]}=(L_1,L_\infty)_{l_p(2^{k}\varphi(2^{-k}))}^K$$
and, taking into account the equality 
\begin{equation}\label{K-funct}
 K(t, f; L_1, L_\infty) = \int_0^t f^*(s)\,ds,\;\;t>0
\end{equation}
\cite[Theorem~5.2.1]{BL}, we infer that
$$
\|x\|_{{\mathcal L}_{[\theta]}}\asymp \left( \int_{0}^{1} \left[ x^{**}(t)
\varphi(t) \right ]^{p} \frac{dt}{t} \right)^{1/p}$$
(as above, $x^{**}(t):=\frac1t \int_0^t x^*(s)\,ds$).
Since $0<\gamma_\varphi\le\delta_\varphi<1,$ from the Hardy inequality \cite[Theorem~2.6.6]{KPS} 
(see also Preliminaries) it follows that ${\mathcal L}_{[\theta]}=\Lambda_{p,\varphi}$.
By Theorem \ref{theor 3}, $\Lambda_{p,\varphi}$
is a $p$-disjointly homogeneous space. Moreover, its uniqueness is an immediate consequence of 
Proposition \ref{Uniqueness1} and the definition of $p$-disjointly homogeneous space.
\end{proof}

\begin{rem}
If $\varphi(t)=t^{1/q},$ $1<q<\infty,$ and $1 \le p < \infty$ 
we obtain the spaces $L_{q,p}$ (see Section 2); the fact that they are $p$-disjointly homogeneous is known
for a long time (see e.g. \cite{FJT} and \cite{CD}).
\end{rem}

\section{Non-ultrasymmetric $p$-disjointly homogeneous r.i. spaces}

Regarding to the result of Theorem \ref{Uniqueness} the following rather natural question appears: 
If the space $\Lambda_{p,\varphi},$
with $1<p<\infty$ and $0<\gamma_\varphi\le\delta_\varphi<1,$ is a unique
$p$-disjointly homogeneous r.i. space with the fundamental function
$\varphi$? In the next theorem that problem is resolved in negative even in the case of
power functions.

\begin{theor}\label{theor 5} Let $1<p<\infty.$
For every increasing concave function $\varphi$ on
$[0,1]$ such that $\lim_{t\to 0}\varphi(t)=0$ and $\beta_\varphi<1$ there exists a $p$-disjointly homogeneous r.i. 
space $Y$ with the fundamental function $\varphi$ such that $Y\not\in Int(M(\tilde{\varphi}),\Lambda(\varphi))$.
\end{theor}

For proving this theorem we need several auxiliary results. The first of them was proved in \cite[Proposition~2]{AM08}. 
In what follows we set $\chi_{k}: = \chi_{(2^{-k},2^{-k+1}]},$ $ k = 1, 2, \ldots$

\begin{lemma} \label{prop: K-method}
Let $X$ be an r.i. space on $[0, 1]$ such that its upper Boyd index $\beta_{X} < 1$. Then, we have
\begin{equation*}
X = \left(L_{1}, L_{\infty} \right)_{E(\varphi(2^{-k}) 2^{k})}^{K},
\end{equation*}
where $\varphi$ is the fundamental function of $X$ and $E$ is the Banach lattice of sequences with the norm
\begin{equation*}
\| (a_{k})\|_{E} = \Big\| \sum_{k=1}^{\infty} a_{k} \frac{\chi_{k}}{\varphi(2^{-k})}
\Big\|_{X}.
\end{equation*}
\end{lemma}

Let as recall that an Orlicz function $M$ on $[0,\infty)$ is called {\it regularly varying 
at $\infty$} if the limit $\lim_{t\to\infty}\frac{M(xt)}{M(t)}$
exists for every $x>0.$ Then, as is well known, we have
$$
\lim_{t\to\infty}\frac{M(xt)}{M(t)}=x^p,$$
with some $1\le p<\infty,$ and then we say that $M$ is {\it regularly varying 
at $\infty$ of order $p$}.

\begin{lemma} 
\label{lemma: regularly varying functions} If $M$
is an Orlicz function regularly varying at $\infty$ of order $p$, the Orlicz space $L_M$ is
$p$-disjointly homogeneous.
\end{lemma}
\begin{proof}
Consider the sets 
$$
E_{M,s}^\infty:=\overline{\left\{\frac{M(xt)}{M(t)}:\,t>s\right\}}\;\;\mbox{and}\;\;
E_{M}^\infty:=\cap_{s>0}E_{M,s}^\infty,$$
where the closure is taken in the space $C[0,1].$
Since $M$ regularly varies at $\infty$ of order $p$, by \cite[Lemma~6.1]{Kal},
there is a certain constant $C>0$ such that for every
$0<x_0\le 1$ it can be found $s_0>0$ such that for all $t\ge s_0$ and
$x_0\le x\le 1$ we have
$$
C^{-1}x^p\le \frac{M(xt)}{M(t)}\le Cx^p.$$
Hence, if $s\ge s_0,$ then arbitrary function $f\in E_{M,s}^\infty$
satisfies the inequality
$$
C^{-1}x^p\le f(x)\le Cx^p\;\;\mbox{for all}\;\;x_0\le x\le 1.$$
Since $x_0>0$ may be chosen arbitrarily small, combining this together with the definition 
of the set $E_{M}^\infty$, we infer that every function from  $E_{M}^\infty$
is equivalent to the function $x^p$ at zero.
Finally, applying \cite[Theorem~4.1]{FHST}, we get desired result.
\end{proof}

\begin{prop} 
\label{prop:  main}
Let $X$ and $Y$ be r.i. spaces on $[0,1]$, $1\le p<\infty.$
Suppose that $Y$ is separable and $X$ is $p$-disjointly homogeneous.
Then, if 
\begin{equation}\label{eq20}
\Big\| \sum_{k=1}^{\infty} a_{k} \frac{\chi_{k}}{\varphi_Y(2^{-k})}\Big 
\|_{Y}\asymp  \Big\| \sum_{k=1}^{\infty} a_{k} \frac{\chi_{k}}{\varphi_X(2^{-k})}\Big 
\|_{X},
\end{equation}
for arbitrary $a_k\in\hj$, then $Y$ is also $p$-disjointly homogeneous.
\end{prop}
\begin{proof}
At first, we prove the property of $p$-disjoint homogeneity for block basic sequences of  
$\{\chi_k\}$ normalized in $Y$, i.e., for an arbitrary sequence $\{y_i\},$ $\|y_i\|_Y=1$, such that
$$
y_i=\sum_{k=n_i}^{n_{i+1}-1} a_k\chi_k,\;\;1\le n_1<n_2<\dots,\;a_k\in\hj.
$$
By assumption, without loss of generality, we can assume that after normalization the sequence
$$
x_i:=\sum_{k=n_i}^{n_{i+1}-1} a_k\frac{\varphi_Y(2^{-k})}{\varphi_X(2^{-k})}\chi_k,\;\;i=1,2,\dots
$$
is equivalent in $X$ to the unit vector basis in $l_p.$
By \eqref{eq20}, for any $c_i\in\hj$ we have
\begin{eqnarray*}
 \Big\|\sum_{i=1}^\infty c_iy_i\Big\|_Y &=& \Big\|\sum_{i=1}^\infty c_i
\sum_{k=n_i}^{n_{i+1}-1} a_k\frac{\varphi_Y(2^{-k})}{\varphi_Y(2^{-k})}\chi_k\Big\|_Y\\
&\asymp& \Big\|\sum_{i=1}^\infty c_i
\sum_{k=n_i}^{n_{i+1}-1} a_k\frac{\varphi_Y(2^{-k})}{\varphi_X(2^{-k})}\chi_k\Big\|_X\\
&=& \Big\|\sum_{i=1}^\infty c_ix_i\Big\|_X 
\asymp \Big(\sum_{i=1}^\infty |c_i|^p\|x_i\|_X^p\Big)^{1/p}.
\end{eqnarray*}

Since
\begin{eqnarray*}
 \|x_i\|_X&=&\Big\|\sum_{k=n_i}^{n_{i+1}-1} a_k\frac{\varphi_Y(2^{-k})}{\varphi_X(2^{-k})}\chi_k\Big\|_X\\
&\asymp& \Big\|\sum_{k=n_i}^{n_{i+1}-1} a_k\chi_k\Big\|_Y=\|y_i\|_Y=1,
\end{eqnarray*}
we come to the equivalence
$$
 \Big\|\sum_{i=1}^\infty c_iy_i\Big\|_Y \asymp \|(c_i)\|_{l_p},$$
and our claim is proved.

Now, show that the general case can be reduced to the previous one.  
A given sequence $\{y_i\}_{i=1}^\infty\subset Y$ of pairwise disjoint functions such that $\|y_i\|_Y=1,$ $i=1,2,\dots,$
we construct a normalized block basic sequence of $\{\chi_k\},$ which is 
equivalent to $\{y_i\}$ in $Y$.

Since $Y$ is a separable r.i. space, we can assume that the functions $y_i$ are finite-valued,
$supp\,y_i=[\alpha_i,\beta_i],$ 
$y_i(t)=y_i^*(t-\alpha_i)$ if $\alpha_i\le t\le \beta_i,$ $i=1,2,\dots$
Setting
$$
z_1':=\sum_{k=1}^\infty y_1^*(2^{-k})\chi_k,$$
it is easily to see that
$$
\sigma_{1/2}z_1'(t)\le y_1^*(t)\le z_1'(t),\;\;0\le t\le 1,$$
where
$\sigma_{1/2}x(t)=x(2t).$
Let $\varepsilon_i>0$ and $\varepsilon_i\to 0$ as $i\to\infty.$
Since $y_1$ is a finite-valued function and $\lim_{t\to +0}\varphi_Y(t)=0$
(as $Y$ is separable), there is $n_1\in\fg$ such that
$$
y_1^*(t)=y_1^*(2^{-n_1})\;\;\mbox{if}\;\;0\le t\le 2^{-n_1}\;\;\mbox{and}\;\;y_1^*(2^{-n_1})\varphi_Y(2^{-n_1})<\varepsilon_1.$$
Then, the functions
$$
u_1:=y_1^*\chi_{(2^{-n_1},1]}\;\;\mbox{and}\;\;z_1:=\sum_{k=1}^{n_1-1} y_1^*(2^{-k})\chi_k$$
satisfy the following conditions:
$$
\|y_1^*-u_1\|_Y<\varepsilon_1$$
and
$$
\sigma_{1/2}z_1(t)\le u_1(t)\le z_1(t)+y_1^*(2^{-n_1})\chi_{n_1},\;\;0\le t\le 1.$$
Clearly, one may assume that $\beta_2-\alpha_2<2^{-n_1}.$ Then $ y_2^*(t)=0$ if $t\ge 2^{-n_1}$, whence
$$
\sigma_{1/2}z_2'(t)\le y_2^*(t)\le z_2'(t),\;\;0\le t\le 1,$$
where we set
$$
z_2':=\sum_{k=n_1+1}^\infty y_2^*(2^{-k})\chi_k.$$
As above, there exists $n_2>n_1$, for which we have
$$
y_2^*(t)=y_2^*(2^{-n_2})\;\;\mbox{if}\;\;0\le t\le 2^{-n_2}\;\;\mbox{and}\;\;y_2^*(2^{-n_2})\varphi_Y(2^{-n_2})<\varepsilon_2.$$
Now, if the functions $u_2$ and $z_2$ are defined as follows:
$$
u_2:=y_2^*\chi_{(2^{-n_2},2^{-n_1}]}\;\;\mbox{and}\;\;z_2:=\sum_{k=n_1+1}^{n_2-1} y_2^*(2^{-k})\chi_k,$$
we obtain
$$
\|y_2^*-u_2\|_Y<\varepsilon_2$$
and
$$
\sigma_{1/2}z_2(t)\le u_2(t)\le z_2(t)+y_2^*(2^{-n_2})\chi_{n_2},\;\;0\le t\le 1.$$
Proceeding in the same way, we find the numbers 
$n_0=1<n_1<n_2<\dots$ and the functions
\begin{equation}\label{eq21}
 u_i:=y_i^*\chi_{(2^{-n_i},2^{-n_{i-1}}]}\;\;\mbox{and}\;\;z_i:=\sum_{k=n_i+1}^{n_{i+1}-1} y_i^*(2^{-k})\chi_k
\end{equation}
such that for all $i=1,2,\dots$ we have
\begin{equation}\label{eq22}
 y_i^*(2^{-n_i})\varphi_Y(2^{-n_i})<\varepsilon_i,
\end{equation}
\begin{equation}\label{eq23}
 \|y_i^*-u_i\|_Y<\varepsilon_i,
\end{equation}
and
\begin{equation}\label{eq24}
 \sigma_{1/2}z_i(t)\le u_i(t)\le z_i(t)+y_i^*(2^{-n_i})\chi_{n_i},\;\;0\le t\le 1.
\end{equation}

Now, set $\bar{u}_i(t)=0$ if $0\le t<\alpha_i$ or $\beta_i <t\le 1,$ and
$\bar{u}_i(t)=u_i(t-\alpha_i)$ if $\alpha_i\le t\le \beta_i$ $(i=1,2,\dots).$
Since the functions $u_i$ (resp. $\bar{u}_i$) are pairwise disjoint and, by \eqref{eq23}, 
$$
\|y_i-\bar{u}_i\|_Y=\|y_i^*-u_i\|_Y<\varepsilon_i,$$
thanks to the well-known principle of small perturbations of basis
(see e.g. \cite[Theorem~1.3.9]{AK}), choosing sufficiently small $\varepsilon_i,$ 
we have
\begin{equation}\label{eq24a}
 \Big\|\sum_{i=1}^\infty c_iy_i\Big\|_Y\asymp  \Big\|\sum_{i=1}^\infty c_iu_i\Big\|_Y,
\end{equation}
for all $c_i\in\hj.$ Moreover, taking into account inequalities \eqref{eq24} and the fact that
the functions $\sigma_{1/2}z_i$ (resp. $z_i+y_i^*(2^{-n_i})\chi_{n_i}$) are pairwise disjoint, we obtain
$$
\sum_{i=1}^\infty c_i \sigma_{1/2}z_i(t)\le \sum_{i=1}^\infty c_i u_i(t)\le
\sum_{i=1}^\infty c_i (z_i(t)+y_i^*(2^{-n_i})\chi_{n_i}(t)),\;\;0\le t\le 1,$$
whence for any $c_i\in\hj$
$$
\Big\|\sigma_{1/2}\Big(\sum_{i=1}^\infty c_iz_i\Big)\Big\|_Y\le \Big\|\sum_{i=1}^\infty c_i u_i\Big\|_Y\le
\Big\|\sum_{i=1}^\infty c_i (z_i+y_i^*(2^{-n_i})\chi_{n_i})\Big\|_Y.$$
Further, since  $\varepsilon_i>0$ may be chosen arbitrarily small, from inequalities
\eqref{eq22}, as above, we infer that
$$
\Big\|\sum_{i=1}^\infty c_i (z_i+y_i^*(2^{-n_i})\chi_{n_i})\Big\|_Y
\asymp \Big\|\sum_{i=1}^\infty c_iz_i\Big\|_Y.$$
Therefore, from the preceding inequality, combined together with the facts that $z=\sigma_{2}(\sigma_{1/2}z)$ 
and $\|\sigma_{2}\|_{Y\to Y}\le 2$ \cite[Corollary~1 after Theorem~2.4.5]{KPS}, it follows that
$$
\Big\|\sum_{i=1}^\infty c_iu_i\Big\|_Y\asymp \Big\|\sum_{i=1}^\infty c_iz_i\Big\|_Y,$$
whence, by \eqref{eq24a}, we obtain
$$
\Big\|\sum_{i=1}^\infty c_iy_i\Big\|_Y\asymp \Big\|\sum_{i=1}^\infty c_iz_i\Big\|_Y.$$
Since $\{z_i\}$ is a block basic sequence of $\{\chi_k\},$ by the first part of the proof (passing
to an appropriate subsequence), we have
$$
\Big\|\sum_{i=1}^\infty c_iz_i\Big\|_Y\asymp \|(c_i)\|_{l_p},$$
which completes the proof.
\end{proof}

\begin{lemma} \label{Orlicz function} [\cite[p.~41]{Kal} and \cite[Example~2]{AM08}]
For every $1<p<\infty$ there exists an Orlicz function $M$,
regularly varying at $\infty$ of order $p,$ such that 
the Orlicz space $L_M\not\in Int(M(\tilde{\psi}),\Lambda(\psi))$,
where $\psi(t)=1/M^{-1}(1/t),$ $0<t\le 1.$ 
%
\end{lemma}
\begin{proof}
Let $v_n \in (0,1),\;v_1 > v_2 > \dots > 0$ and $v_n \rightarrow 0$. The function 
$g: {\mathbb R} \rightarrow (0, \infty)$ is defined as follows: 
$g(t) = p$ for $t \leq 1$, and $g(t) = p + (-1)^n v_n$
if $2^{n-1} < t \leq 2^n, n = 1, 2, \dots$. Denote
$$
f(u) = \int_0^u g(t) dt ~~{\rm and} ~~M(u) = e^{f(\log u)} ~(u > 0).
$$ 
Then, it is not hard to check that $M$ satisfies all required conditions
(see also \cite[Example~2]{AM08}).
\end{proof}

Now, let $1<p<\infty$ and let an Orlicz function $M$ satisfy
the conditions of Lemma \ref{Orlicz function}.
Clearly, we can assume that $M(1)=1.$
It is not hard to check that $\alpha_{L_M}=\beta_{L_M}=\gamma_\psi=\delta_\psi=1/p$ (see e.g. \cite[p.~139]{LT2}). 
Hence, since $p>1$, by Lemma \ref{prop: K-method}, we have
\begin{equation}\label{eq100}
 L_M = \left(L_{1}, L_{\infty} \right)_{G(\psi(2^{-k}) 2^{k})}^{K},
\end{equation}
where $G$ is the Banach lattice of sequences with the norm
\begin{equation}\label{eq101}
\| (a_{k})\|_{G} = \Big\| \sum_{k=1}^{\infty} a_{k} \frac{\chi_{k}}{\psi(2^{-k})} 
\Big\|_{L_M}.
\end{equation}

Denote by $e_k$ $(k=1,2,\dots)$ the vectors of the unit basis in sequence spaces and by
$P_i,$ $i=1,2,\dots,$ the shift operators, which are defined as follows:
$$
P_i\Big(\sum_{k=1}^\infty a_ke_k\Big)=\sum_{k=1}^\infty a_{k+i}e_k.$$

\begin{lemma} \label{space G}
For arbitrary $\varepsilon>0$ there is $B=B(\varepsilon)>0$ such that
$$
\|P_i\|_{G\to G}\le B2^{\varepsilon i},\;\;i=1,2,\dots$$
\end{lemma}
\begin{proof}
From the definition of $G$ it follows that
$$
\| (a_{k})\|_{G}=\inf\left\{\lambda>0:\,\sum_{k=1}^{\infty} M\left(\frac{|a_{k}|M^{-1}(2^k)}{\lambda}\right)2^{-k}\le 1\right\}.
$$ 
Therefore, if $\|(a_{k})\|_{G}=1$, we have
\begin{equation}\label{eq200}
 \sum_{k=1}^{\infty} M(|a_{k}|M^{-1}(2^k))2^{-k}\le 1.
\end{equation}
Let us estimate 
\begin{equation}\label{eq25}
 \|P_i (a_{k})\|_{G}=\inf\left\{\lambda>0:\,\sum_{j=i+1}^{\infty} M\left(\frac{|a_{j}|M^{-1}(2^{j-i})}{\lambda}\right)2^{i-j}\le 1\right\}
\end{equation}
for each $i=1,2,\dots$

First, since $\gamma_\psi=1/p$, for arbitrary $0<\eta<1/p$ there is $C_1=C_1(\eta)>0$ such that
$$
\sup_{0<s\le 1}\frac{\psi(st)}{\psi(s)}\le C_1t^{1/p-\eta},\;\;0<t\le 1,$$
or equivalently (because of $\psi(t)=1/M^{-1}(1/t)$),
\begin{equation}\label{eq26}
u^{1/p-\eta}M^{-1}(v)\le C_1 M^{-1}(uv),\;\;u,v\ge 1.
\end{equation}
Similarly, as $\delta_\psi=1/p$, it can be found $C_2=C_2(\eta)>0$, for which
$$
\sup_{0<s\le 1,t\ge 1,st\le 1}\frac{\psi(st)}{\psi(s)}\le C_2t^{1/p+\eta},\;\;t\ge 1,$$
or 
$$
u^{1/p+\eta}M^{-1}(v)\le C_2 M^{-1}(uv),\;\;u\le 1,v\ge 1, uv\ge 1.
$$
Hence, for arbitrary $v\ge z\ge 1$ 
$$
\frac{M^{-1}(v)}{v^{1/p+\eta}}\le C_2\frac{M^{-1}(z)}{z^{1/p+\eta}},$$
and, taking into account that $M$ increases and $M(1)=1,$ we obtain
\begin{equation}\label{eq27}
M(y)\le C_2\left(\frac{y}{x}\right)^{\frac{p}{p\eta+1}}M(x)\;\;\mbox{if}\;\;x\ge y\ge 1.
\end{equation}


Denote by $S$ the set of all positive integers $j\ge i+1$ such that
$|a_{j}|M^{-1}(2^{j-i})C_1^{-1}C_2^{-1}>1.$ Then, for every $j\in S$
from \eqref{eq26}, \eqref{eq27}, and the convexity 
of the function $M$ it follows that
\begin{eqnarray*}
 M\left(\frac{|a_{j}|M^{-1}(2^{j-i})}{C_1C_2}\right) &\le&
M\left(\frac{|a_{j}|2^{\frac{i(p\eta-1)}{p}}M^{-1}(2^{j})}{C_2}\right)\\
 &\le& C_22^{\frac{i(p\eta-1)}{1+p\eta}} M\left(\frac{|a_{j}|M^{-1}(2^j)}{C_2}\right)\\
&=& 2^{(\varepsilon -1) i}M(|a_{j}|M^{-1}(2^j)),
\end{eqnarray*}
where $\varepsilon:=2p\eta/(p\eta +1)$ can be made arbitrarily small together with $\eta.$
Therefore, by \eqref{eq200}, we have
\begin{eqnarray*}
\sum_{j=i+1}^{\infty} M\left(\frac{|a_{j}|M^{-1}(2^{j-i})}{C_1C_2}\right)2^{i-j} &=&
\sum_{j\in S} M\left(\frac{|a_{j}|M^{-1}(2^{j-i})}{C_1C_2}\right)2^{i-j}\\
&+&\sum_{j\not\in S} M\left(\frac{|a_{j}|M^{-1}(2^{j-i})}{C_1C_2}\right)2^{i-j}\\
&\le& 2^{\varepsilon i}\sum_{j=i+1}^{\infty}M(|a_{j}|M^{-1}(2^j))2^{-j}+\sum_{j=1}^{\infty}2^{-j}\le 2\cdot 2^{\varepsilon i}.
\end{eqnarray*}
Finally, using the convexity of $M$ once more and formula \eqref{eq25}, we obtain 
$$
\|P_i (a_{k})\|_{G} \le 2^{\varepsilon i}\cdot 2C_1C_2\|(a_j)\|_G,$$
that is, desired estimate holds with $B:=2C_1C_2.$
\end{proof}

\begin{proof}[Proof of Theorem \ref{theor 5}]
If $M$ is an Orlicz function satisfying the conditions of 
Lemma \ref{Orlicz function}, the Orlicz space $L_M$ is not ultrasymmetric (see Preliminaries). 
Then, since $0<\gamma_\psi=\delta_\psi=1/p<1,$ by \cite[Theorem 2.1]{Pust}, 
from \eqref{eq100} and \eqref{eq101} it follows that
the Banach lattice $G$ is not symmetric. At the same time, $G$ is separable together with 
$L_M$. Moreover, applying Lemma \ref{lemma: regularly varying functions}, we see that
$L_M$ is a $p$-disjointly homogeneous space.

Now, we define the space $Y$ as follows:
\begin{equation}\label{I:P2b}
 Y: = \left(L_{1}, L_{\infty} \right)_{G(\varphi(2^{-k}) 2^{k})}^{K}.
\end{equation}
Since $\|e_k\|_G=1,$ $k=1,2,\dots$, we have $l_1\subset G\subset l_\infty,$
whence $\Lambda(\varphi)\subset Y\subset M(\tilde{\varphi}).$
Moreover, $Y\not\in Int(M(\tilde{\varphi}),\Lambda(\varphi))$ (equivalently, $Y$ is not ultrasymmetric), because of the
Banach lattice $G$ is not symmetric \cite[Theorem 2.1]{Pust}, and $Y$ is separable together with $G$.
Let us check that $\beta_Y<1.$

By definition of the real interpolation spaces and formula \eqref{K-funct}, for arbitrary $i\in\fg$ we have
\begin{eqnarray*}
\|\sigma_{2^i} x\|_Y &=& \Big\|\Big(\int_0^{2^{-k}}x^*(2^{-i}s)\,ds\cdot\varphi(2^{-k})2^k\Big)_{k=1}^\infty\Big\|_G\\
&=& \Big\|\Big(2^{i+k}\int_0^{2^{-i-k}}x^*(s)\,ds\cdot\varphi(2^{-k})\Big)_{k=1}^\infty\Big\|_G.
\end{eqnarray*}
Furthermore, since $\delta_\varphi<1$, for some $C>0$ and $\eta>0$
$$
\sup_{0<s\le 1,st\le 1}\frac{\varphi(st)}{\varphi(s)}\le Ct^{1-\eta},\;\;t\ge 1.$$
Therefore, by Lemma \ref{space G}, with $\varepsilon=\eta/2$, we obtain
\begin{eqnarray*}
 \|\sigma_{2^i} x\|_Y &\le& C2^{i(1-\eta)}\Big\|\Big(2^{i+k}\int_0^{2^{-(i+k)}}x^*(s)\,ds\cdot\varphi(2^{-(i+k)})\Big)_{k=1}^\infty\Big\|_G\\
&\le& BC2^{i(1-\eta/2)}\Big\|\Big(2^{k}\int_0^{2^{-k}}x^*(s)\,ds\cdot\varphi(2^{-k})\Big)_{k=1}^\infty\Big\|_G=BC2^{i(1-\eta/2)}\|x\|_Y,
\end{eqnarray*}
whence $\beta_Y<1.$

Thus, we may apply Lemma \ref{prop: K-method} and thereby get
\begin{equation}\label{I:P2a}
Y = \left(L_{1}, L_{\infty} \right)_{E(\{\varphi(2^{-k}) 2^{k}\}_{k=1}^{\infty})}^{K},
\end{equation}
where $E$ is the Banach sequence lattice with the norm
\begin{equation*}
\| (a_{k})\|_{E} = \Big\| \sum_{k=1}^{\infty} a_{k} \frac{\chi_{k}}{\varphi_Y(2^{-k})} 
\Big\|_{Y}.
\end{equation*}
Comparing \eqref{I:P2b} with \eqref{I:P2a}, by Proposition~2 from the paper
\cite{A3}, we conclude that $E=G$ (with equivalence of norms), i.e., 
$$
\Big\| \sum_{k=1}^{\infty} a_{k} \frac{\chi_{k}}{\varphi_Y(2^{-k})} 
\Big\|_{Y}\asymp \Big\| \sum_{k=1}^{\infty} a_{k} \frac{\chi_{k}}{\psi(2^{-k})} 
\Big\|_{L_M}.
$$
Since $L_M$ is a $p$-disjointly homogeneous, from Proposition 
\ref{prop:  main} it follows that $Y$ also has this property, and the proof is completed.
\end{proof}

\section{Sequences of pairwise disjoint functions in the spaces $(M(\tilde{\varphi}),\Lambda(\varphi))_{\theta,p}$}

If the dilation indices of an increasing concave function $\varphi$ are non-trivial, 
i.e., $0<\gamma_\varphi\le\delta_\varphi<1,$ then applying the reiteration theorem 
\cite[Theorem~4.3.1]{BK} precisely in the same way as in the proof of Theorem
\ref{Uniqueness}, it can be shown that, for every $1<p<\infty$, the real interpolation space
$(M(\tilde{\varphi}),\Lambda(\varphi))_{1/p,p}$ coincides with the complex interpolation
space $[M(\tilde{\varphi}),\Lambda(\varphi)]_{1/p}$ (and hence
with $\Lambda_{p,\varphi}$).
Thus, in this case the space $[M(\tilde{\varphi}),\Lambda(\varphi)]_{1/p}$
is $p$-disjointly homogeneous. In this section we investigate some properties
of sequences of pairwise disjoint functions in the spaces
$(M(\tilde{\varphi}),\Lambda(\varphi))_{\theta,p},$ $0<\theta<1,$ $1\le p<\infty,$
for an arbitrary increasing concave function $\varphi.$

Our first result, a version of the well-known Levy theorem on the existence
of $l_p$-subspaces in the real interpolation spaces $(X_0,X_1)_{\theta,p}$,  
\cite{Levy} or \cite[Theorem~4.6.22]{BK} (in a more general setting, see
\cite{Mast} and \cite {A1}),
concerns even with a more general case when the couple 
$(M(\tilde{\varphi}),\Lambda(\varphi))$ is replaced with a couple $(X_0,X_1)$
such that $X_1\subset X_0.$ Given $0<\theta<1$ and $1\le p<\infty,$ we denote
$X_{\theta,p}:=(X_0,X_1)_{\theta,p}.$

%


\begin{theor}\label{theor 2}
Suppose that pairwise disjoint functions $x_n\in X_{\theta,p}$, 
$\|x_n\|_{\theta,p}=1$ $(n=1,2,\dots)$ and $\lim_{n\to\infty}{\|x_n\|_{X_0}}=0.$

Then for arbitrary $\varepsilon>0$ there is a subsequence $\{z_j\}\subset \{x_n\}$ such that
for all $\lambda_j\in\hj$ $(j=1,2,\dots)$ we have
\begin{equation}\label{eq10.5a}
(1-\varepsilon)\left(\sum_{j=1}^{\infty}|\lambda_j|^p\right)^{1/p}\le
\Big\|\sum_{j=1}^{\infty}\lambda_jz_j\Big\|_{\theta,p}\le
(1+\varepsilon)\left(\sum_{j=1}^{\infty}|\lambda_j|^p\right)^{1/p}.
\end{equation}
%
%
\end{theor}

The main ingredient in the proof of Theorem \ref{theor 2} is the following assertion.

\begin{prop}\label{theor 1}
Let $0<\theta<1,$ $1\le p<\infty$ and let $K$ be a compact subset of
the unit sphere in $X_{\theta,p}.$ Moreover, suppose that 
$x_k\in X_{\theta,p},$ $\|x_k\|_{\theta,p}=1$ $(k=1,2,\dots)$ and 
$\lim_{k\to\infty}{\|x_k\|_{X_0}}=0.$
Then, for arbitrary $\varepsilon\in (0,1)$ there exists 
$k_0\in \mathbb{N}$ such that 
\begin{equation}\label{eq3a}
(1-\varepsilon)(|\lambda_1|^p+|\lambda_2|^p)^{1/p}\le
\|\lambda_1x+\lambda_2x_{k_0}\|_{\theta,p}\le(1+\varepsilon)(|\lambda_1|^p+|\lambda_2|^p)^{1/p}
\end{equation}
for each $x\in K$ and $\lambda_1,\lambda_2\in\hj.$
\end{prop}

In the proof of this proposition we will make use of the following
technical lemmas proved in \cite{Mast}.

\begin{lemma}\label{Ma1}
Let $a\ge 0,$ $\varepsilon>0$ and $1\le p<\infty.$ Then the
following inequality holds:
$$
(1+a\varepsilon)^p+(a+\varepsilon)^p\le (1+\varepsilon)^p(1+a^p).$$
\end{lemma}

\begin{lemma}\label{Ma2}
Let $0<\varepsilon<1,$ $1\le p<\infty$ and
$c:=(1-\varepsilon^p/2^p)^{1/p}.$ Then the following inequalities
hold:
$$
(a)\;\;(1-\varepsilon)^p(1+a^p)\le
(ac-\varepsilon/2)^p+(c-a\varepsilon/2)^p\;\;\mbox{for all}\;\;a\in
[\varepsilon/(2c),2c/\varepsilon]$$ and
$$
(b)\;\;(1-\varepsilon)^p(1+a^p)\le |1-a|^p\;\;\mbox{for all}\;\;a\in
[0,\varepsilon/(2c)]\cup [2c/\varepsilon,\infty).$$
\end{lemma}

\begin{proof}[Proof of Proposition \ref{theor 1}]
Denote by $S$ the sublinear operator defined on $X_0$ as follows: 
$$
Sx=\{Sx\}_n,\;\;\;\{Sx\}_n:=2^{n\theta} K(2^{-n}, x; X_0,X_1),\;\;n=1,2,\dots$$
Taking into account inequalities \eqref{eq99} and equality \eqref{eq1000}, we see that
$S$ is an isometric embedding from $X_{\theta,p}$ into $l_p.$ Therefore, by assumption,
$S(K)$ is a compact set in $l_p$, and a given $\varepsilon\in (0,1)$
there exists $m\in\mathbb{N}$ such that
\begin{equation}\label{eq4a}
\sum_{n=m+1}^\infty \{Sx\}_{n}^p<\left(\frac{\varepsilon}{2}\right)^p\;\;\mbox{for
each}\;\;x\in K. 
\end{equation}
Hence, taking into account that $K$ is a subset of
the unit sphere in $X_{\theta,p},$ we obtain
\begin{equation}\label{eq5}
\sum_{n=1}^m \{Sx\}_{n}^p\ge 1-\left(\frac{\varepsilon}{2}\right)^p\;\;\mbox{for
each}\;\;x\in K.
\end{equation}

Further, since the functional $\{Sx\}_n$, for each $n\in\mathbb{N}$, is equivalent to
the norm in $X_0$ and $\lim_{k\to\infty}\|x_k\|_{X_0}=0,$
for every $n\in\mathbb{N}$ we have $\{Sx_k\}_n\to 0$ as $k\to\infty$. 
Thus, we can find $k_0\in\mathbb{N}$ such that the element $y:=x_{k_0}$ satisfies
the condition: 
\begin{equation}\label{eq6}
\sum_{n=1}^m \{Sy\}_{n}^p<\left(\frac{\varepsilon}{2}\right)^p.
\end{equation}
Combining this together with the fact that $\|y\|_{\theta,p}=1$, we also have
\begin{equation}\label{eq7}
\sum_{n=m+1}^\infty \{Sy\}_{n}^p\ge 1-\left(\frac{\varepsilon}{2}\right)^p.
\end{equation}

Now, we prove the right-hand side inequality in \eqref{eq3a}.
Clearly, it suffices to consider the case when $\lambda_1:=\lambda$
is arbitrary and $\lambda_2=1.$ Firstly, from \eqref{eq4a} and \eqref{eq6} it follows that
$$
\Big(\sum_{n=m+1}^\infty \{S(\lambda x+y)\}_n^{p}\Big)^{1/p}\le
|\lambda|\Big(\sum_{n=m+1}^\infty \{Sx\}_n^{p}\Big)^{1/p}+\|y\|_{\theta,p}\le |\lambda|\varepsilon +1$$
and
$$
\Big(\sum_{n=0}^m \{S(\lambda x+y)\}_n^{p}\Big)^{1/p}\le
|\lambda|\|x\|_{\theta,p}+\Big(\sum_{n=0}^m \{Sy\}_n^{p}\Big)^{1/p}\le |\lambda|+\varepsilon.$$
Then,
\begin{eqnarray*}
 \|\lambda x+y\|_{\theta,p}&=&\left(\sum_{n=0}^m \{S(\lambda x+y)\}_n^{p}
+\sum_{n=m+1}^\infty \{S(\lambda x+y)\}_n^{p}\right)^{1/p}\\
&\le& \left((|\lambda|+\varepsilon)^p+(1+|\lambda|\varepsilon)^p\right)^{1/p},
\end{eqnarray*}
whence, by Lemma \ref{Ma1}, we obtain 
$$
\|\lambda x+y\|_{\theta,p}\le
(1+\varepsilon)(1+|\lambda|^p)^{1/p},$$ and the right-hand side
inequality in \eqref{eq3a} is proved.

To prove the left-hand side inequality in \eqref{eq3a} assume, firstly,
that $|\lambda|\in [\varepsilon/(2c),2c/\varepsilon],$ where
$c:=(1-\varepsilon^p/2^p)^{1/p}.$ By \eqref{eq5} and \eqref{eq6},
$$
\Big(\sum_{n=0}^m \{S(\lambda x+y)\}_n^{p}\Big)^{1/p}\ge
|\lambda|\Big(\sum_{n=0}^m \{Sx\}_n^{p}\Big)^{1/p}-\Big(\sum_{n=0}^m \{Sy\}_n^{p}\Big)^{1/p}\ge
|\lambda|c-\frac{\varepsilon}{2},$$ and, by \eqref{eq4a} and
\eqref{eq7},
$$
\Big(\sum_{n=m+1}^\infty \{S(\lambda x+y)\}_n^{p}\Big)^{1/p}\ge
\Big(\sum_{n=m+1}^\infty \{Sy\}_n^{p}\Big)^{1/p}-|\lambda|\Big(\sum_{n=m+1}^\infty \{Sx\}_n^{p}\Big)^{1/p}\ge
c-|\lambda|\frac{\varepsilon}{2}.$$ These inequalities together with
Lemma \ref{Ma2}(a) yield
\begin{eqnarray*}
\|\lambda x+y\|_{\theta,p}&=&\left(\sum_{n=0}^m \{S(\lambda x+y)\}_n^{p}
+\sum_{n=m+1}^\infty \{S(\lambda x+y)\}_n^{p}\right)^{1/p}\\
&\ge&
\left(\left(|\lambda|c-\frac{\varepsilon}{2}\right)^p+\left(c-|\lambda|\frac{\varepsilon}{2}\right)^p
\right)^{1/p}\ge (1-\varepsilon)(1+|\lambda|^p)^{1/p}.
\end{eqnarray*}
If
$|\lambda|\not\in [\varepsilon/(2c),2c/\varepsilon],$ then Lemma
\ref{Ma2}(b) implies that
$$
(1-\varepsilon)(1+|\lambda|^p)^{1/p}\le |1-|\lambda|\,|\le \|\lambda
x+y\|_{\theta,p},$$ and the proof of the left-hand side inequality in
\eqref{eq3a} is complete.
\end{proof}

\begin{proof}[Proof of Theorem \ref{theor 2}]
For a given $\varepsilon>0,$ let $\theta_j\in (0,1)$ be such that
\begin{equation}\label{eq11}
1-\varepsilon<\prod_{j=1}^\infty(1-\theta_j)<
\prod_{j=1}^\infty(1+\theta_j)<1+\varepsilon.
\end{equation}
We construct inductively a subsequence $\{z_j\}\subset \{x_n\}$
such that, for every $l\in\fg$ and all $\lambda_1,\dots,\lambda_{l}\in\hj,$ we have
\begin{equation}\label{eq12}
\prod_{j=1}^{l}(1-\theta_j)\left(\sum_{j=1}^{l}|\lambda_j|^p\right)^{1/p}\le
\Big\|\sum_{j=1}^{l}\lambda_jz_j\Big\|_{\theta,p}\le
\prod_{j=1}^{l}(1+\theta_j)\left(\sum_{j=1}^{l}|\lambda_j|^p\right)^{1/p}.
\end{equation}

Let $z_1:= x_{n_1}$ be arbitrary. Then, applying Proposition \ref{theor 1}
to the one-element set $\{z_1\}$, we can find $z_2 := x_{n_2},$ $n_2>n_1,$ such that 
$$
\prod_{j=1}^{2}(1-\theta_j)(|\lambda_1|^p+|\lambda_2|^p)^{1/p}\le
\|\lambda_1z_1+\lambda_2z_2\|_{\theta,p}\le \prod_{j=1}^{2}(1+\theta_j)(|\lambda_1|^p+|\lambda_2|^p)^{1/p},
$$ for every $\lambda_1,\lambda_2\in\hj.$

Assume that we already have chosen the elements $z_1,\dots,z_{l-1}$ from the sequence $\{x_n\},$
which satisfy the inequality 
\begin{equation}\label{eq12.1}
\prod_{j=1}^{l-1}(1-\theta_j)\left(\sum_{j=1}^{l-1}|\lambda_j|^p\right)^{1/p}\le
\Big\|\sum_{j=1}^{l-1}\lambda_jz_j\Big\|_{\theta,p}\le
\prod_{j=1}^{l-1}(1+\theta_j)\left(\sum_{j=1}^{l-1}|\lambda_j|^p\right)^{1/p},
\end{equation}
for all $\lambda_1,\dots,\lambda_{l-1}\in\hj.$ It is clear that the set
$K:=\{x\in{\rm span}\,\{z_1,\dots,z_{l-1}\}:\,\|x\|_{\theta,p}=1\}$
is compact in $X_{\theta,p}.$ Therefore, by Proposition \ref{theor 1}, there
exists a function $z_l:= x_{n_l},$ $n_l>n_{l-1},$ such that 
\begin{equation}\label{eq13}
(1-\theta_l)(1+|\gamma|^p)^{1/p}\le \|x+\gamma
z_l\|_{\theta,p}\le(1+\theta_l)(1+|\gamma|^p)^{1/p},
\end{equation}
for all $x\in K$ and $\gamma\in\hj.$ Let
$\lambda_1,\dots,\lambda_{l}\in\hj$ be arbitrary. Clearly, it may be
assumed that $\|\sum_{j=1}^{l-1}\lambda_jz_j\|_{\theta,p}>0.$
Setting
$$
\gamma:=\lambda_l\cdot\|\sum_{j=1}^{l-1}\lambda_jz_j\|_{\theta,p}^{-1},\;\;
x:=(\sum_{j=1}^{l-1}\lambda_jz_j)\cdot\|\sum_{j=1}^{l-1}\lambda_jz_j\|_{\theta,p}^{-1}$$
(note that $x\in K)$ and using \eqref{eq13}, we infer that
$$
(1-\theta_l)\left(\Big\|\sum_{j=1}^{l-1}\lambda_jz_j\Big\|_{\theta,p}^{p}+|\lambda_l|^p\right)^{1/p}\le
\Big\|\sum_{j=1}^{l}\lambda_jz_j\Big\|_{\theta,p}\le(1+\theta_l)\left(\Big\|\sum_{j=1}^{l-1}\lambda_jz_j\Big\|_{\theta,p}^{p}
+|\lambda_l|^p\right)^{1/p}.$$ 

Combining the latter inequality with \eqref{eq12.1} we obtain
\eqref{eq12}. Since inequality \eqref{eq10.5a} is an immediate consequence 
\eqref{eq12} and \eqref{eq11}, the proof is complete.
\end{proof}

\begin{rem}\label{rem: auxiliary}
Repeating the arguments from the proof of Theorem 1
in \cite{A1}, it can be shown also that a sequence $\{z_j\}\subset \{x_n\}$
satisfying \eqref{eq10.5a}, in addition, spans in $X_{\theta,p}$
a $(1+\varepsilon)$--complemented subspace (that is, there is a projection
$P$ bounded in $X_{\theta,p}$, $\|P\|\le 1+\varepsilon$, such that $P(X_{\theta,p})=[z_j]$).
\end{rem}

Now, let us consider a somewhat different situation, where we return to the special case
of couples $(M(\tilde{\varphi}),\Lambda(\varphi))$. Let ${\mathcal L}_{\theta,p}:= 
(M(\tilde{\varphi}),\Lambda(\varphi))_{\theta,p}$, $0<\theta<1,$ $1\le p<\infty.$

\begin{prop} 
\label{lemma: auxiliary} 
Let $1<p<\infty$, and let an increasing concave function $\varphi$
satisfy condition \eqref{eq1001}. Suppose that 
$x_n\in \Lambda(\varphi),$ $n=1,2,\dots,$ are pairwise disjoint functions with
$\|x_n\|_{\Lambda(\varphi)}\asymp \|x_n\|_{M(\tilde{\varphi})}\asymp 1$ $(n\in\fg).$
Then, there exists a subsequence of $\{x_n\}$ equivalent to the unit vector basis
in $l_p,$
which spans a subspace complemented in ${\mathcal L}_{\theta,p}$.
\end{prop}
\begin{proof}
As above, passing to a subsequence, we can assume that
\begin{equation}\label{eq18}
 \Big\|\sum_{k=1}^\infty a_kx_k\Big\|_{\Lambda(\varphi)}\asymp \|(a_k)\|_{l_1}\;\;\mbox{and}\;\;
 \Big\|\sum_{k=1}^\infty a_kx_k\Big\|_{M(\tilde{\varphi})}\asymp \|(a_k)\|_{l_\infty}.
\end{equation}
Since $\|x_n\|_{M(\tilde{\varphi})}\asymp 1$ $(n\in\fg)$ and
$M(\psi)'=\Lambda(\psi)$ for every increasing concave function $\psi$ \cite[Theorem~4.5.4]{KPS},
we can find $x_n'\in\Lambda(\tilde{\varphi})$ such that $\|x_n'\|_{\Lambda(\tilde{\varphi})}\asymp 1,$
${\rm supp}\,x_n'\subset {\rm supp}\,x_n,$ and
\begin{equation}\label{eq16}
 \int_0^1 x_n x_n'\,dt=1,\;\;n\in\fg.
\end{equation}
We claim that
\begin{equation}\label{eq1500}
\|x_n'\|_{M({\varphi})}\asymp 1,\;\;n\in\fg.
\end{equation}
In fact, on the one hand, it is clear that $\|x_n'\|_{M({\varphi})}\le \|x_n'\|_{\Lambda(\tilde{\varphi})}\le C,$
$n=1,2,\dots$ for some constant $C>0.$ Further, since $\|x_n\|_{\Lambda(\varphi)}\asymp 1$,  by \eqref{eq16},
we have $\inf_{n\in\fg}\|x_n'\|_{M({\varphi})}>0.$

Now, let us check that the projection
$$
Px(t)=\sum_{k=1}^\infty \int_0^1 x x_n'\,ds\cdot x_n(t),\;\;0\le t\le 1,$$
is bounded both in $\Lambda(\varphi)$ and $M(\tilde{\varphi})$.

First, for arbitrary set $e\subset [0,1]$, by \eqref{eq18}, we have
$$
\|P\chi_e\|_{\Lambda(\varphi)}\asymp \sum_{k=1}^\infty \Big|\int_0^1 \chi_e x_k'\,ds\Big|
=\sum_{k=1}^\infty \alpha_k\int_e x_k'\,ds,$$
where $\alpha_k:={\rm sign}\,\int_e x_k'\,ds$, $k=1,2,\dots$
In view of \eqref{eq1500}, as above, it can be assumed that
$$
\Big\|\sum_{k=1}^\infty b_kx_k'\Big\|_{M({\varphi})}\asymp \|(b_k)\|_{l_\infty}.
$$
Hence, since the functions $x_k'$, $k=1,2,\dots,$ are pairwise disjoint, by definition of
the norm in a Marcinkiewicz space, we obtain
\begin{eqnarray*}
 \|P\chi_e\|_{\Lambda(\varphi)} &\le& C\int_e\Big|\sum_{k=1}^\infty \alpha_k x_k'\Big|\,ds
=\frac{C}{\varphi(m(e))}\int_e\Big|\sum_{k=1}^\infty \alpha_k x_k'\Big|\,ds\cdot{\varphi(m(e))}\\
&\le& C\Big\|\sum_{k=1}^\infty \alpha_k x_k'\Big\|_{M({\varphi})}\|\chi_e\|_{\Lambda(\varphi)}\le 
C\|\chi_e\|_{\Lambda(\varphi)},
\end{eqnarray*}
and the boundedness of the projection $P$ in $\Lambda(\varphi)$ is proved. Furthermore, applying
\eqref{eq18} and the fact that $\Lambda({\psi})^*=M({\psi})$  
\cite[Theorem~4.5.2]{KPS}, we have 
$$
\|Px\|_{M(\tilde{\varphi})}\asymp \sup_{k=1,2,\dots}\Big|\int_0^1 xx_k'\,ds\Big|\le 
\sup_{k=1,2,\dots}\|x_k'\|_{\Lambda(\tilde{\varphi})}\|x\|_{M(\tilde{\varphi})}\le C\|x\|_{M(\tilde{\varphi})},$$
which implies the boundedness of $P$ in the space $M(\tilde{\varphi})$.

Thus, $(P(M(\tilde{\varphi})),P(\Lambda(\varphi)))$ is a
complemented subcouple of the Banach couple $(M(\tilde{\varphi}),\Lambda(\varphi))$.
At the same time, by \eqref{eq18}, the couple $(P(M(\tilde{\varphi})),P(\Lambda(\varphi)))$ is isomorphic to the couple
$(l_\infty,l_1)$. Therefore, by the well-known result of Baouendi and Goulaouic
\cite[Theorem 1]{BG} (see also \cite[Theorem 1.17.1]{Trib}), the sequence
$\{x_k\}$ is equivalent in ${\mathcal L}_{\theta,p}$ to the unit vector basis in the space
$(l_\infty,l_1)_{\theta,p}.$ In particular, if $\theta=1/p$, we have $(l_\infty,l_1)_{\theta,p}=l_p$ (see e.g. 
\cite[Theorem~5.2.1]{BL}); hence, $\{x_k\}$ is equivalent in ${\mathcal L}_{\theta,p}$ to the unit vector basis in $l_p.$
Finally, since the projection $P$ is bounded in ${\mathcal L}_{\theta,p}$, the subspace 
$[x_k]$ is complemented in the latter space.
\end{proof}

From Theorem \ref{theor 2}, Remark \ref{rem: auxiliary} and Proposition \ref{lemma: auxiliary}
we obtain

\begin{cor}
Let $1<p<\infty$ and let an increasing concave function $\varphi$
satisfy condition \eqref{eq1001}. Suppose that 
$x_n\in X_{1/p,p},$ $n=1,2,\dots,$ are pairwise disjoint functions such that either
$\lim_{n\to\infty}\|x_n\|_{M(\tilde{\varphi})}=0$ or
$\|x_n\|_{\Lambda(\varphi)}\asymp \|x_n\|_{M(\tilde{\varphi})}\asymp 1$ $(n\in\fg).$
Then, there exists a subsequence of $\{x_n\},$ which is equivalent to the unit vector basis
in $l_p$
and spans a subspace complemented in ${\mathcal L}_{\theta,p}$. 
\end{cor}

Regarding to the latter statement the following natural question arises:
Let $1<p<\infty$, and let an increasing concave function $\varphi$
satisfy condition \eqref{eq1001}. Suppose that 
$x_n\in \Lambda(\varphi),$ $n=1,2,\dots,$ are pairwise disjoint functions such that 
$\|x_n\|_{\Lambda(\varphi)}\to\infty$ as $n\to\infty$ and 
$\|x_n\|_{M(\tilde{\varphi})}\asymp 1$ $(n\in\fg).$ Is there a subsequence of $\{x_n\}$ 
equivalent to the unit vector basis in $l_p$ ?
The positive answer to this question would mean that
each of the spaces $X_{1/p,p}=(M(\tilde{\varphi}),\Lambda(\varphi))_{1/p,p},$ under condition \eqref{eq1001},
is $p$-disjointly homogeneous.

Though we are not able to resolve the above question we can present
some weaker result in the positive direction.

\begin{prop} 
\label{lemma: disjoint functions} 
Let an increasing concave function $\varphi$
satisfy condition \eqref{eq1001}. Suppose that 
$x_n\in {\mathcal L}_{\theta,p}=(M(\tilde{\varphi}),\Lambda(\varphi))_{\theta,p},$ $n=1,2,\dots,$ 
where $0<\theta<1$ and $1\le p<\infty,$ are arbitrary pairwise disjoint functions.

Then, the closed linear span $[x_n]$ in ${\mathcal L}_{\theta,p}$ is not closed in the Marcinkiewicz space
$M(\tilde{\varphi}).$
\end{prop}
\begin{proof}
 Without loss of generality, we can assume that 
\begin{equation}
\label{eq1}
\|x_n\|_\vr=1,\;\;n=1,2,\dots
\end{equation}
Then, since $\lim_{t\to 0}{t}/{\varphi(t)}=0,$ applying once more
\cite[Theorem~5.1]{FJT} or \cite[Proposition~1]{Tok} (and passing
to a subsequence if necessary) , we obtain  
$$
\Big\|\sum_{k=1}^\infty a_kx_k\Big\|_{M(\tilde{\varphi})}\asymp \|(a_k)\|_{l_\infty}.$$
Assuming the contrary, let the subspace $[x_n]$ be closed in $M(\tilde{\varphi}).$ Then combining the latter
equivalence with the fact that the space ${\mathcal L}_{\theta,p}$ has the Fatou property, we infer that
\begin{equation}
\label{eq3}
z:=\sum_{k=1}^\infty x_k \in {\mathcal L}_{\theta,p}.
\end{equation}

In order to obtain a contradiction with \eqref{eq3}, let us estimate the
$K$-functional $K(t,z;M(\tilde{\varphi}),\Lambda(\varphi))$ at points $t=2^{-k}$, $k=1,\dots,$ from below.
By definition, for every $k=1,2,\dots$ there exist
$u_i^k\in \Lambda(\varphi)$ and $v_i^k\in M(\tilde{\varphi})$  $(i=1,2,\dots)$ such that 
$u_i^k+v_i^k=x_i$, ${\rm supp}\,u_i^k\subset {\rm supp}\,x_i$,
${\rm supp}\,v_i^k\subset {\rm supp}\,x_i$ and
$$
K(2^{-k},z;M(\tilde{\varphi}),\Lambda(\varphi))\ge \frac12 (\|v^k\|_\vr+2^{-k}\|u^k\|_{\il}),$$
where $u^k:=\sum_{i=1}^\infty u_i^k\in \Lambda(\varphi),$ $v^k:=\sum_{i=1}^\infty v_i^k\in M(\tilde{\varphi}),$ $k=1,2,\dots$
From the estimate
\begin{eqnarray*}
\|z\|_{\theta,p}^p &\ge& \sum_{k=1}^\infty\left(K(2^{-k},z;M(\tilde{\varphi}),\Lambda(\varphi)) 2^{k\theta}\right)^p\\
&\ge& 2^{-p}\sum_{k=1}^\infty\left(\|v^k\|_\vr+2^{-k}\|u^k\|_{\il}\right)^p2^{k\theta p}\\
&\ge& 2^{-p}\sum_{k=1}^\infty 2^{kp\theta}\|v^k\|_\vr^p
 \end{eqnarray*}
and \eqref{eq3} it follows that $\|v^k\|_\vr\to 0$ as $k\to\infty.$
Therefore, for some $k_0\in\fg$ we have
$$
\|v_i^{k_0}\|_\vr\le \|v^{k_0}\|_\vr<\frac{1}{2},\;\; i=1,2,\dots$$
Combining the latter inequality with the equality $u_i^k+v_i^k=x_i$ and \eqref{eq1}, we obtain
\begin{equation}
\label{eq4}
\|u_i^{k_0}\|_\il\ge \|u_i^{k_0}\|_\vr\ge 1-\|v_i^{k_0}\|_\vr>\frac{1}{2},\;\;i=1,2,\dots
\end{equation}

Let $\varepsilon_k>0$ $(i=1,2,\dots),$ $\sum_{k=1}^\infty \varepsilon_k<\infty.$ 
Since $\lim_{t\to 0}\varphi(t)=0,$ by \cite[Lemma~1]{Tok}, there is $\delta=\delta(\varepsilon_1)>0$ such that for arbitrary
$y\in\il,$ with $m({\rm supp}\, y)<\delta,$ we have
$$
\|u_1^{k_0}+y\|_\il\ge \|y\|_\il+(1-\varepsilon_1)\|u_1^{k_0}\|_\il.$$
Then, if $i_0=1$ and $i_1>i_0$ is chosen so that $m({\rm supp}\, u_{i_1}^{k_0})<\delta,$
from the preceding inequality and \eqref{eq4} it follows that
$$
\|u_{i_0}^{k_0}+ u_{i_1}^{k_0}\|_\il\ge \frac{1}{2}(1+(1-\varepsilon_1)).$$
Applying \cite[Lemma~1]{Tok} once more, but now to the function
$u_{i_0}^{k_0}+ u_{i_1}^{k_0}$ and $\varepsilon_2$ in the similar way, we can find $i_2>i_1,$
for which
\begin{eqnarray*}
 \|u_{i_0}^{k_0}+ u_{i_1}^{k_0}+u_{i_2}^{k_0}\|_\il &\ge& \|u_{i_2}^{k_0}\|_\il+
\frac{1}{2}(1-\varepsilon_2)(1+(1-\varepsilon_1))\\
&\ge&  \frac{1}{2}(1+(1-\varepsilon_1)+(1-\varepsilon_1)(1-\varepsilon_2)).
\end{eqnarray*}
Arguing in the same way, we construct an increasing sequence of indices
$(i_j)_{j=0}^\infty$ such that 
$$
\Big\|\sum_{j=0}^\infty u_{i_j}^{k_0}\Big\|_\il\ge 
\frac{1}{2}\left(1+\sum_{r=1}^n\prod_{k=1}^r(1-\varepsilon_k)\right),\;\;n\in\fg.$$
Since $\sum_{k=1}^\infty \varepsilon_k<\infty,$ we have $\prod_{k=1}^\infty(1-\varepsilon_k)=\eta>0.$
Therefore, from the preceding inequality it follows that
$$
\|u^{k_0}\|_\il\ge \Big\|\sum_{j=0}^\infty u_{i_j}^{k_0}\Big\|_\il\ge \frac{\eta (n+1)}{2},\;\;\mbox{for every}\;\;n\in\fg.$$
Clearly, this contradicts to the fact that $u^{k_0}\in\il,$ and the proof is completed.
\end{proof}

The following theorem is an immediate consequence of Proposition \ref{lemma: disjoint functions} 
and arguments used in the proof of Theorem \ref{theor 2}.

\begin{theor}\label{theor 10}
Let an increasing concave function $\varphi$ satisfy condition \eqref{eq1001}. Suppose that 
$x_n\in {\mathcal L}_{\theta,p}=(M(\tilde{\varphi}),\Lambda(\varphi))_{\theta,p},$ $n=1,2,\dots,$ 
where $0<\theta<1$ and $1\le p<\infty,$ are arbitrary pairwise disjoint functions,
$\|x_n\|_{\theta,p}=1$ $(n=1,2,\dots).$

Then, for every $\varepsilon>0$ there exists a block basis 
$u_i=\sum_{j=n_i}^{n_{i+1}-1} a_jx_j,$ $1=n_0<n_1<\dots,$ satisfying the following conditions:

(i) for arbitrary $\lambda_k\in\hj,$ $k=1,2,\dots,$ we have
\begin{equation}\label{eq10.5}
(1-\varepsilon)\left(\sum_{k=1}^{\infty}|\lambda_k|^p\right)^{1/p}\le
\Big\|\sum_{k=1}^{\infty}\lambda_ku_k\Big\|_{\theta,p}\le
(1+\varepsilon)\left(\sum_{k=1}^{\infty}|\lambda_k|^p\right)^{1/p};
\end{equation}

(ii) the closed linear span $[u_k]$ is $(1+\varepsilon)$--complemented in the space ${\mathcal L}_{\theta,p}.$
\end{theor}


\begin{center}
S.V. Astashkin \\
Department of Mathematics and Mechanics\\
Samara State University\\
443011 Samara, Acad. Pavlov, 1 \\
Russian Federation\\
e-mail: {\it astash@samsu.ru}
\vskip 0.4cm
\end{center}


\begin{thebibliography}{}

\bibitem{AK} F. Albiac and N. J. Kalton, {\it Topics in Banach Space Theory}, Graduate Texts in Mathematics 233,
Springer, New York, 2006.

\bibitem {A3} S. V. Astashkin, {\it A description of interpolation spaces between
$(l_1(w^0),l_1(w^1))$ and $(l_{\infty}(w^0),l_{\infty}(w^1)),$}
Mat. Zametki {\bf 35}(1984), no.~4, 497--503 (in Russian); English transl.
in Math. Notes {\bf 35}(1984), no.~4, 261--265.

\bibitem {A3} S. V. Astashkin, {\it A property of functors of the real interpolation method,} Mat. Zametki
{\bf 38}(1985), no. 3, 393–-406 (in Russian); English transl. in Math. Notes
{\bf 38}(1985), no. 3, 725–-732.

\bibitem {A2} S.~V. Astashkin, {\it On cones of step functions in symmetric spaces,} Sibirsk. Mat. Zh. {\bf 34}(1993),
no. 4, 7–-16 (in Russian); English transl. in Siberian Math. J. {\bf 34}(1993),
no. 4, 597–605. {\bf 34}(1993), no. 4, 7--16.

\bibitem {A1} S.~V. Astashkin, {\it Geometrical properties of Banach spaces generated by
sublinear operators,} Positivity {\bf 17}(2013), 223--234.

\bibitem{AM08} S. V. Astashkin and L. Maligranda, {\it Ultrasymmetric Orlicz spaces,}
J. Math. Anal. Appl. {\bf 347} (2008), no. 2, 273--285.

\bibitem{BG} M. S. Baouendi and C. Goulaouic, {\it Commutation de l'intersection et des foncteurs d'interpolation}, 
C. R. Acad. Sci. Paris S\'er. A-B 26 (1967), 313--315.

\bibitem {BS} C. Bennett and R. Sharpley, {\it Interpolation of 
Operators}, Academic Press, Boston 1988.

\bibitem {BL}  J. Bergh and J. L{\"o}fstr{\"o}m, {\it Interpolation Spaces.
An Introduction}, Springer-Verlag, Berlin-New York, 1976.

\bibitem {BK} Yu.~A.~Brudnyi and N.~Ya.~Kruglyak, {\it Interpolation Functors
and Interpolation Spaces 1.} North-Holland, 1991.

\bibitem{CD} N. L. Carothers and S. J. Dilworth, {\it Subspaces of $L_{p,q}$}, Proceedings of the American
Mathematical Society {\bf 104} (1988), 537–545.

\bibitem{DKO} V.I. Dmitriev, S.G. Krein and V.I. Ovchinnikov, {\it Fundamentals of the theory of interpolation 
of linear operators}. Geometry of linear spaces and operator theory, pp. 31–74 (in Russian). 
Jaroslav. Gos. Univ., Yaroslavl', 1977. 
 
\bibitem{FJT} T. Figiel, W. B. Johnson and L. Tzafriri,
 \textit{On Banach lattices and spaces having local unconditional structure 
 with applications to Lorentz function spaces.}
 J. Approx. Theory {\bf 13}(1975), 395--412.

\bibitem {FHST} J. Flores, F.L. Hernandez, E.M. Semenov, and P. Tradacete, {\it Strictly singular
and power-compact operators on Banach lattices}, Israel J. Math. {\bf 188} (2012), 323--352.

\bibitem{FHSTT} J. Flores, F. L. Hernandez, E. Spinu, P. Tradacete, and V. G. Troitsky,
\textit{Disjointly homogeneous Banach lattices: duality and complementation}.
 J. Funct. Anal. (2014), http://dx.doi.org/j.jfa.2013.12.024.

\bibitem{FTT2} J. Flores, P. Tradacete, and V. G. Troitsky,
\textit{Disjointly homogeneous Banach latices and compact product of operators.}
 J. Math. Anal. Appl. {\bf 354} (2009), 657--663.
 
 \bibitem{HST}  F. L. Hernandez, E. M. Semenov, and P. Tradacete,
\textit{Rearrangement invariant spaces with Kato property} (Preprint).

\bibitem {Kal} N. J. Kalton, {\it Calder{\'o}n couples of rearrangement invariant spaces}, Studia Math. 
{\bf 106} (1993), no. 3, 233--277.

 \bibitem{KPS} S. G. Krein, Yu. I. Petunin, and E. M. Semenov,
\textit{Interpolation of linear operators.} Translations of Mathematical Monographs, 
54. American Mathematical Society, 1982.

\bibitem{Levy} M.~Levy, {\it L'espace d'interpolation r\'{e}el $(A_0,A_1)_{\theta,p}$
conteint $l^p$}, C. R. Acad. Sci. Paris, \textbf{289}~(1979), 675--677.
 
\bibitem{LT2} Lindenstrauss J., Tzafriri L. {\it Classical Banach Spaces II. Function spaces.} 
Berlin-Heidelberg-New York: Springer-Verlag, 1979.

\bibitem {Lo} G. G.  Lorentz, {\it On the theory of spaces 
$\Lambda$}, Pacific J. Math. {\bf 1}(1951), 411--429. 

\bibitem {Loz} G. Ja. Lozanovskii, {\it Certain Banach lattices}, Sibirsk. Mat. Zh. {\bf 10} (1969),
584–599 (in Russian).

\bibitem{Mast} M.~Masty{\l}o, \emph{Banach spaces via sublinear operators}, Math.
Japon. \textbf{36} (1991), no. 1, 785--792.

\bibitem {Pust} E. Pustylnik, {\it Ultrasymmetric spaces}, J. London Math. 
Soc. (2) {\bf 68}(2003), no. 1, 165--182. 

\bibitem{Pust1} E. Pustylnik, {\it Sobolev type inequalities in ultrasymmetric 
spaces with applications to Orlicz-Sobolev embeddings}, J. Funct. Spaces Appl. 
{\bf 3}(2005), no. 2, 183--208.

\bibitem{Pust2} E. Pustylnik, {\it Ultrasymmetric sequence spaces in 
approximation theory}, Collect. Math. {\bf 57}(2006), no. 3, 257--277.

\bibitem{Ra} Y. Raynaud, {\it On Lorentz-Sharpley spaces}, Israel Math. Conf. 
Proc., Vol. 5, Bar-Ilan Univ., Ramat Gan 1992, 207--228.

\bibitem{Sh} R. Sharpley, {\it Spaces $\Lambda_{\alpha}(X)$ and 
interpolation}, J. Funct. Anal. {\bf 11}(1972), 479--513.

\bibitem {Shes} V. A: Shestakov, {\it Transformations of Banach ideal spaces and interpolation
of linear operators}, Bull. Acad. Polon. Sci. S\'{e}r. Sci. Math. {\bf 29} (1981),
 no. 11-12, 569–-577 (1982) (in Russian).

\bibitem {Tok} Tokarev E.~V. {\it On subspaces of some symmetric spaces,}
Teoriya funkcii, funkc. anal. i ih prilozh., {\bf 24}(1975), 156--161 (in Russian).

\bibitem{Trib} H. Triebel, {\it Interpolation Theory. Function Spaces. Differential Operators},
VEB Deutscher Verlag der Wissenschaften, Berlin 1978.

%

\end{thebibliography}
\end{document}